\documentclass{amsart}
\usepackage{enumerate}
\usepackage{longtable}
\usepackage[ruled]{algorithm2e}

\newtheorem{theorem}{Theorem}[section]

\newtheorem{proposition}[theorem]{Proposition}
\newtheorem{corollary}[theorem]{Corollary}

\theoremstyle{definition}
\newtheorem{definition}[theorem]{Definition}
\newtheorem{example}[theorem]{Example}

\theoremstyle{remark}

\numberwithin{equation}{section}

\newcommand{\A}{{\mathcal A}}
\newcommand{\B}{{\mathcal B}}
\newcommand{\C}{{\mathcal C}}
\newcommand{\vC}{{\sf C}}
\newcommand{\D}{{\frak D}}
\newcommand{\F}{{\mathcal F}}
\newcommand{\G}{{\mathcal G}}
\renewcommand{\O}{{\mathcal O}}
\newcommand{\R}{{\mathcal R}}
\newcommand{\fS}{{\frak S}}

\newcommand{\bbZ}{{\mathbb Z}}

\begin{document}

\title[Universal Cycles for Minimum Coverings and 2-Radius Sequences]{Universal Cycles for Minimum Coverings of Pairs by Triples, with Application to
2-Radius Sequences}

\author{Yeow Meng Chee}
\address{Division of Mathematical Sciences \\
School of Physical and Mathematical Sciences \\
Nanyang Technological University \\
Singapore}
\thanks{Research  
supported in part by the National Research Foundation of Singapore
under Research Grant NRF-CRP2-2007-03 and by the Nanyang Technological University
under Research Grant M58110040.}

\author{San Ling}
\author{Yin Tan}
\author{Xiande Zhang}





\begin{abstract}
A new ordering, extending the notion of universal cycles of Chung {\em et al.} (1992), is proposed
for the blocks of $k$-uniform set systems. 
Existence of minimum coverings of pairs by triples that possess such
an ordering is established for all orders. Application to the construction of short
2-radius sequences
is given, with some new
2-radius sequences found through computer search.
\end{abstract}

\maketitle


\section{Introduction}

Determining the existence of orderings on the blocks of various classes
of set systems (all terms are defined in the next section) to meet specified
criteria is a fundamental problem in discrete mathematics,
due to applications in combinatorial computing.
The following two types of orderings have long histories:

\begin{enumerate}[(i)]
\item An ordering of the blocks such that any two successive blocks have
symmetric difference of the smallest possible size 
(size of one for set systems having more than one block size, and
size of two for $k$-uniform set systems).
This type of ordering was first studied by Gray \cite{Gray:1953}
and has now come to be known as {\em combinatorial Gray codes} \cite{Savage:1997}. 
Combinatorial Gray codes are known to exist for the complete
set system $(X,2^X)$ \cite{Gray:1953},
for the complete $k$-uniform set system $(X,\binom{X}{k})$ 
\cite{Bitneretal:1976,EadesMcKay:1984,Ruskey:1988},
and for some classes of triple systems with index two
\cite{Dewar:2007}.
\item An ordering of the blocks such that any two successive blocks have nonempty
intersection. Such an ordering
is also equivalent to a hamiltonian path (or hamiltonian cycle, if one also insists
that the first block and last block have nonempty intersection)
in the block intersection graph of the set system. 
The existence question for this type of ordering for Steiner triple systems was first raised by
Ron Graham in 1987, at an American Mathematical Society meeting. 
It is known that such orderings exist for Steiner 2-designs \cite{HorakRosa:1988},
for pairwise balanced designs whose maximum block size is at most twice the
minimum block size \cite{Alspachetal:1990},
for pairwise balanced designs whose minimum block size at least three \cite{Hare:1995},
and for triple systems of arbitrary index \cite{Horaketal:1999}.
\end{enumerate}

Recently, Dewar \cite{Dewar:2007} studied universal cycles for block designs.
This is an ordering of the blocks such that two successive blocks differ in a small
structural way. The concept of universal cycles was introduced by
Chung {\em et al.} \cite{Chungetal:1992} as a generalization of
de Bruijn sequences \cite{deBruijn:1946}. Universal cycles are known to exist
for the complete $k$-uniform set system $(X,\binom{X}{k})$ when $k\in\{2,3,4,5,6\}$,
provided $k$ divides $\binom{|X|-1}{k-1}$ and $|X|$ is sufficiently large
\cite{Chungetal:1992,Jackson:1993,Hurlbert:1994,Jackson:unpublished}, and
for some classes of triple systems of index two \cite{Dewar:2007}.

Under the definition of universal cycles of Chung {\em et al.} \cite{Chungetal:1992},
universal cycles cannot exist for $k$-uniform set systems in which there are two 
blocks intersecting in less than $k-1$ points. To overcome this restriction, 
Dewar \cite{Dewar:2007} proposed an extension to the definition of universal cycles
and showed that with this new definition, there exist
universal cycles for some classes of Steiner 2-designs, including Steiner triple systems.
However, these new universal cycles of Dewar has the disadvantage that 
the blocks of a set system cannot always be recovered from its universal cycle.

In this paper, the concept of $s$-shift universal cycles is proposed
as another natural extension to the universal cycles of Chung {\em et al.} \cite{Chungetal:1992}.
Minimum $(n,3,2)$-coverings for which there exist $2$-shift universal cycles are constructed
for all $n\geq 3$, by considering alternating hamiltonian cycles
in their block intersection graphs. An application to the construction of 
2-radius sequences of order $n$
\cite{JaromczykLonc:2004} is given.

\section{Notations, Terminology, and Known Results}

The ring $\bbZ/m\bbZ$ is denoted by $\bbZ_m$, and 
the set of nonnegative integers is denoted $\bbZ_{\geq 0}$.
For $n$ a positive integer, the set $\{1,2,\ldots,n\}$ is denoted by $[n]$.
For $X$ a finite set and $k$ an integer, $0\leq k\leq |X|$,
the set of all $k$-subsets of $X$ is denoted by $\binom{X}{k}$.

A {\em set system} is a pair $\fS=(X,\A)$, where $X$ is a finite set and $\A\subseteq 2^X$.
The members of $X$ are called {\em vertices} or {\em points}, and the members of $\A$
are called {\em edges} or {\em blocks}. The {\em order} of $\fS$ is the number of vertices 
$|X|$, and the {\em size} of $\fS$ is the number of edges $|\A|$.
Let $K\subseteq\bbZ_{\geq 0}$. 
The set system $\fS$ is said to be {\em $K$-uniform} if 
$|A|\in K$ for all $A\in\A$. We often write $\{k\}$-uniformity
as $k$-uniformity. A $2$-uniform set system is
also known simply as a {\em graph}.

An {\em edge-colored graph} is a pair $(\Gamma,\chi)$, where 
$\Gamma=(X,\A)$ is a graph and $\chi$ is a function on $\A$ (called a {\em coloring}).
The image of $\chi$, $\vC=\{\chi(A):A\in\A\}$, is called the {\em color set}.
Elements of the color set are called {\em colors}. An edge $A\in\A$ is said to 
have color $c\in\vC$ if $\chi(A)=c$. A cycle in an edge-colored graph is {\em colorful}
if for every color, there exists some edge in the cycle having that color.
An {\em alternating cycle} in an edge-colored graph
is a cycle  
where adjacent edges have distinct colors.

A {\em pairwise balanced design} (PBD), or more specifically a PBD$(n,K)$, is a $K$-uniform
set system $(X,\A)$ of order $n$ such that every $T\in\binom{X}{2}$
is contained in exactly one block of $\A$. A PBD$(n,\{3\})$ is called
a {\em Steiner triple system of order $n$}, and is denoted by STS$(n)$.
It is well known (see, for example, \cite{ColbournRosa:1999})
that an STS$(n)$ exists if and only if $n\equiv 1$ or $3\pmod{6}$.

A {\em covering of pairs by triples} of order $n$
(or {\em $(n,3,2)$-covering}) is a $3$-uniform set system $(X,\A)$ of order $n$
such that every $T\in\binom{X}{2}$ is contained in at least one block of $\A$.
The minimum size of an $(n,3,2)$-covering is the {\em covering number}
$C(n,3,2)$. An $(n,3,2)$-covering of size $C(n,3,2)$ is said to be {\em minimum}.
Fort and Hedlund \cite{FortHedlund:1958} determined that for all $n\geq 3$,
\begin{eqnarray*}
C(n,3,2)&=&\left\lceil \frac{n}{3}\left\lceil \frac{n-1}{2}\right\rceil\right\rceil.
\end{eqnarray*}
Note that an STS$(n)$ is a minimum $(n,3,2)$-covering.

Let $(X,\A)$ be a set system, and let $\G=\{G_1,G_2,\ldots,G_s\}$ be a partition of $X$ into
subsets, called {\em groups}. The triple $(X,\G,\A)$ is a {\em group divisible design} (GDD)
when every 2-subset of $X$ not contained in a group appears in exactly one block, and
$|A\cap G|\leq 1$ for all $A\in\A$ and $G\in\G$. We denote a GDD $(X,\G,\A)$
by $K$-GDD if $(X,\A)$ is $K$-uniform. The {\em type} of a GDD $(X,\G,\A)$
is the multiset $[|G| : G\in\G]$. When more convenient, the exponentiation notation is used
to describe the type of a GDD: a GDD of type $g_1^{t_1}g_2^{t_2}\ldots g_s^{t_s}$
is a GDD where
there are exactly $t_i$ groups of size $g_i$, $i\in[s]$. A PBD$(n,K)$ can be regarded
as a $K$-GDD
of type $1^n$, where each group contains a single point.

The existence of classes of PBDs and GDDs required in this paper is given below.

\begin{theorem}[Gronau, Mullin, and Pietsch \cite{Gronauetal:1995}]
\label{GMP}
There exists a {\rm PBD}$(n,\{3$, $4,5,6,8\})$, for all $n\geq 3$.
\end{theorem}

\begin{theorem}[Lenz \cite{Lenz:1984}]
\label{Lenz}
There exists a {\rm PBD}$(n,\{4,5,6,7\})$, for all $n\geq 4$,
except when $n\in\{8,9,10,11,12,14,15,18,19,23\}$.
\end{theorem}

\begin{theorem}[Colbourn, Hoffman, and Rees \cite{Colbournetal:1992b}]
\label{CHR}
Let $g,t,u\in\bbZ_{\geq 0}$.
There exists a $\{3\}$-{\rm GDD} of type
$g^tu^1$ if and only if the following conditions are all satisfied:
\begin{enumerate}[{\rm (i)}]
\item if $g>0$ then $t\geq 3$, or $t=2$ and $u=g$, or $t=1$ and $u=0$, or $t=0$;
\item $u\leq g(t-1)$ or $gt=0$;
\item $g(t-1)+u\equiv 0\pmod{2}$ or $gt=0$;
\item $gt\equiv 0\pmod{2}$ or $u=0$;
\item $g^2\binom{t}{2}+gtu\equiv 0\pmod{3}$.
\end{enumerate}
\end{theorem}

\begin{theorem}[Brouwer, Schrijver, and Hanani \cite{Brouweretal:1977}]
\label{BSH}
There exists a $\{4\}$-{\rm GDD} of type $g^t$ if and only if $t\geq 4$ and
\begin{enumerate}[{\rm (i)}]
\item $g\equiv 1$ or $5\pmod{6}$ and $t\equiv 1$ or $4\pmod{12}$; or
\item $g\equiv 2$ or $4\pmod{6}$ and $t\equiv 1\pmod{3}$; or
\item $g\equiv 3\pmod{6}$ and $t\equiv 0$ or $1\pmod{4}$; or
\item $g\equiv 0\pmod{6}$,
\end{enumerate}
with the two exceptions of types $2^4$ and $6^4$, for which $\{4\}$-{\rm GDD}s do not exist.
\end{theorem}

Given a set system $\fS=(X,\A)$ or GDD $\fS=(X,\G,\A)$,
the {\em block intersection graph} of $\fS$ is a loopless 
multigraph $\Gamma_\fS=(Y,\B)$ such that $Y=\A$ and 
there exists $\lambda$ edges between distinct $A,A'\in Y$ if and only if
$|A\cap A'|=\lambda$.
Let $\chi$ be a function that assigns to each edge between $A,A'\in Y$ a
distinct color in $A\cap A'$.
Then $(\Gamma_\fS,\chi)$
is called the {\em edge-colored block intersection graph} of $\fS$.
A set system or GDD is {\em colorful alternating hamiltonian} (c.a.h.) if its edge-colored
block intersection graph has a c.a.h. cycle. For brevity, a
c.a.h. cycle in the edge-colored block intersection graph
of a set system or GDD is simply referred to as a c.a.h. cycle in the set system or GDD.

\section{$s$-Shift Universal Cycles}

Let $\F$ be a set of combinatorial objects, each of ``rank'' $r$, such that each
$F\in\F$ is specified by a sequence $\langle x_1,x_2,\ldots,x_r\rangle$,
where $x_i\in X$, for some fixed $X$.

\begin{definition}[Universal Cycle \cite{Chungetal:1992}]
$U=(u_0,u_1,\ldots,u_{|\F|-1})$ is a {\em universal cycle} for $\F$ if
$\langle u_{i+1}, u_{i+2},\ldots,u_{i+r}\rangle$, $i\in\bbZ_{|\F|}$, runs
through each element of $\F$ exactly once.
\end{definition}

Given a $k$-uniform set system $\fS=(X,\A)$, it is natural to ask if there exists a universal
cycle for $\fS$, that is, a cycle $(x_0,x_1,\ldots,x_{|\A|-1})$ such that
$\{x_{i+1},x_{i+2},\ldots$, $x_{i+k}\}$ runs through $\A$ exactly once, for $0\leq i<|\A|$. Notice
that for $0\leq i<j<|\A|$, $\{x_{i+1},x_{i+2},\ldots,x_{i+k}\}$ 
and $\{x_{j+1},x_{j+2},\ldots,x_{j+k}\}$ intersect in
$k-(j-i)$ points. In particular, for every $i\in\{0,1,\ldots,k-1\}$,
$\A$ must contain a pair of blocks intersecting in $i$ points. This rules out the
existence of universal cycles for $k$-uniform set systems in which there are no
pairs of blocks that intersect in $j$ points, for some $j\in\{0,1,\ldots,k-1\}$.
In particular, there cannot exist universal cycles for 
Steiner triple systems. 

To overcome the stringent condition for existence of universal cycles
for set systems, Dewar \cite{Dewar:2007}
relaxed the condition by allowing a universal sequence for a $k$-uniform
set system $\fS$ to just generate
some representation of $\fS$, instead of generating all the blocks of $\fS$.
The specific representation considered by Dewar is defined as follows.
Let $\fS=(X,\A)$ be a set system. A set $\R\subseteq\binom{X}{r}$ is said to {\em represent}
$\fS$ if every block in $\A$ contains exactly one element of $\R$, and
every $R\in\R$ is contained in exactly one block of $\A$.

\begin{example}
\label{sts7a}
Consider the {\rm STS}$(7)$ whose blocks are $\{1,3,7\}$, $\{1,5,6\}$, $\{4,5,7\}$, $\{2,6,7\}$,
$\{3,4,6\}$, $\{1,2,4\}$, and $\{2,3,5\}$. This set system can be
represented by 
$R=\{ \{1,3\}$, $\{1,5\}$, $\{5,7\}$, $\{6,7\}$, $\{4,6\}$, $\{2,4\}$, $\{2,3\}\}$, for which there exists 
a universal cycle $(3,1,5,7,6,4,2)$.
\end{example}

One disadvantage with Dewar's approach is that the set system
may not be recoverable from a given universal cycle of its representation, as seen
in the example below.

\begin{example}
\label{sts7b}
Consider the {\rm STS}$(7)$ whose blocks are $\{1,3,4\}$, $\{2,3,6\}$, $\{2,4,7\}$, $\{4,5,6\}$,
$\{1,6,7\}$, $\{3,5,7\}$, and $\{1,2,5\}$. This set system is distinct from that in
Example {\rm \ref{sts7a}}, but has the same representation
$R=\{ \{1,3\}$, $\{1,5\}$, $\{5,7\}$, $\{6,7\}$, $\{4,6\}$, $\{2,4\}$, $\{2,3\}\}$.
\end{example}

Here, the notion of universal cycles is extended in another direction.

\begin{definition}[$s$-Shift Universal Cycle]
Let $s$ be a positive integer.
$U=(u_0$, $u_1,\ldots,u_{s|\F|-1})$ is an {\em $s$-shift universal cycle} for $\F$ if
$\langle u_{si+1},u_{si+2},\ldots,u_{si+r}\rangle$, $i\in\bbZ_{|\F|}$, runs
through each element of $\F$ exactly once.
\end{definition}

A $1$-shift universal cycle is equivalent to the universal cycle of 
Chung {\em et al.} \cite{Chungetal:1992}.

\begin{example}
$U=(1,3,7,2,6,4,3,5,2,1,4,7,5,6)$ is a $2$-shift universal cycle for the {\rm STS}$(7)$
in Example {\rm \ref{sts7a}}. The blocks of the {\rm STS}$(7)$ can be recovered from $U$.
\end{example}

The next result gives the equivalence between certain shift universal cycles
and alternating hamiltonian cycles.

\begin{proposition}
For $k\geq 2$,
a $(k-1)$-shift universal cycle for a $k$-uniform set system $\fS$ is equivalent to an
alternating hamiltonian cycle in $\fS$.
\end{proposition}

\begin{proof}
Suppose $U=(u_0,u_1,\ldots,u_{m(k-1)-1})$ is a $(k-1)$-shift universal cycle 
for a $k$-uniform set system $\fS$. Let $A_i = \{u_{i(k-1)},u_{i(k-1)+1},\ldots,u_{i(k-1)+(k-1)}\}$,
$i\in\bbZ_{m}$. Then $\{A_0,A_1\},\{A_1,A_2\},\ldots,\{A_{m-1},A_0\}$ are edges of an
alternating hamiltonian cycle in $\fS$, noting that the color of edge $\{A_i,A_{i+1}\}$ is
$u_{(i+1)(k-1)}$, $i\in\bbZ_m$. Hence $\{A_i,A_{i+1}\}$ and $\{A_{i+1},A_{i+2}\}$ cannot 
possibly be of the same color since $u_{(i+1)(k-1)}$ and $u_{(i+2)(k-1)}$ are both contained
in the $A_{i+1}$.

Conversely, suppose that $\A=\{A_0,A_1,\ldots,A_{m-1}\}$ is the set of blocks
of a $k$-uniform set system $\fS$, and that 
$\{A_0,A_1\},\{A_1,A_2\},\ldots,\{A_{m-1},A_0\}$
are edges of an alternating hamiltonian cycle in $\fS$.
For $i\in\bbZ_m$, order the points within $A_i$ so that 
\begin{enumerate}[(i)]
\item its first point is the color of the edge $\{A_{i-1},A_i\}$, and
\item its last point is the color of the edge $\{A_i,A_{i+1}\}$.
\end{enumerate}
According to this order, construct the sequence obtained by listing down the points of $A_0,A_1,\ldots,A_{m-1}$, with the condition that only one of the last point of $A_i$
and the first point of $A_{i+1}$ is included in the listing. This sequence is a $k$-shift
universal cycle for $\fS$.
\end{proof}

In the next section, the existence of minimum $(n,3,2)$-coverings that admit
$2$-shift universal cycles, is established
for all $n\geq 3$, by considering their alternating hamiltonian cycles.

\section{Alternating Hamiltonian Cycles in Edge-Colored Block Intersection Graphs}
\label{AH}

Some useful recursive constructions are first described.

\subsection{Recursive Constructions} 

We begin with a simple observation.

\begin{proposition}
\label{C+C}
Let $(\Gamma,\chi)$ be an edge-colored block intersection graph,
and let $C_1$ and $C_2$ be two
(vertex) disjoint alternating cycles, of length $m_1$ and $m_2$, respectively.
If there is an edge in $C_1$ of the same color as some edge in $C_2$, then
$\Gamma$ has an alternating cycle of length $m_1+m_2$.
\end{proposition}

\begin{proof}
Suppose $A=\{a,b\}$ and $B=\{c,d\}$ are edges in $C_1$ and $C_2$, respectively, such
that $\chi(A)=\chi(B)$. The $(C_1\setminus \{a,b\}) \cup (C_2\setminus\{c,d\}) \cup
\{\{a,c\},\{b,d\}\}$ is an alternating cycle of length $m_1+m_2$.
\end{proof}

\begin{proposition}[Filling in Groups]
\label{fillin}
If there exists a c.a.h. $\{3\}$-{\rm GDD} of type $[g_1,g_2,\ldots,g_t]$, and
for each $i\in[t]$, there exists a c.a.h. $(g_i,3,2)$-covering, then there exists a
c.a.h. $(\sum_{i=1}^t g_i,3,2)$-covering.
\end{proposition}

\begin{proof}
Let $\D=(X,\G,\A)$ be a c.a.h. $\{3\}$-GDD of type $[g_1,g_2,\ldots,g_t]$, with $C$
as a c.a.h. cycle in $\D$.
For each group $G\in\G$, let $\D_G=(G,\B_G)$ be a c.a.h. $(|G|,3,2)$-covering, with $C_G$
as a c.a.h. cycle in $\D_G$. It is clear that $\D^*=(X,\A\cup(\cup_{G\in\G}\B_G))$ is
a $(\sum_{i=1}^t g_i,3,2)$-covering. Each of $C_G$ contains an edge of the same color
as some edge in $C$, so these can be combined with $C$ via Proposition \ref{C+C} to
give one c.a.h. cycle for $\D^*$.
\end{proof}

The following is another useful construction for c.a.h. $(n,3,2)$-coverings from GDDs.

\begin{proposition}[Adjoining $y$ Points and Filling in Groups]
\label{adjoin}
Let $y\in\bbZ_{\geq 0}$. Suppose there exists a (master) c.a.h. $\{3\}$-{\rm GDD} of type
$[g_1,g_2,\ldots,g_t]$, and suppose the following (ingredients) also exist:
\begin{enumerate}[{\rm (i)}]
\item a c.a.h. $(g_t+y,3,2)$-covering,
\item a c.a.h. $\{3\}$-{\rm GDD} of type $1^{g_i} y^1$, for $i\in[t-1]$.
\end{enumerate}
Then there exists a c.a.h. $(y+\sum_{i=1}^t g_i,3,2)$-covering.
\end{proposition}

\begin{proof}
Let $(X,\G,\A)$ be a c.a.h. $\{3\}$-GDD of type $[g_1,g_2,\ldots,g_t]$, with
$\G=\{G_1,G_2$, $\ldots,G_t\}$ such that $|G_i|=g_i$, $i\in[t]$. Let $Y$ be a set
of size $y$, disjoint from $X$. Let $(G_t\cup Y,\B_t)$ be a c.a.h. $(g_t+y,3,2)$-covering
and $(G_i\cup Y,\{\{x\}:x\in G_i\}\cup\{Y\},\B_i)$ be a c.a.h. $\{3\}$-GDD of type
$1^{g_i} y^1$, for $i\in[t-1]$. Then the set system
\begin{equation*}
\fS = (X\cup Y, \A\cup(\cup_{i=1}^t \B_i))
\end{equation*}
is a $(y+\sum_{i=1}^t g_i,3,2)$-covering. To show that $\fS$ is c.a.h., mimic
the proof of Proposition \ref{fillin}.
\end{proof}

The following construction is similar to that in Proposition \ref{adjoin}, except that we end up
with a GDD with smaller groups instead of an $(n,3,2)$-covering.

\begin{proposition}[Adjoining $y$ Points and Breaking Up Groups]
\label{break}
Let $y\in\bbZ_{\geq 0}$. Suppose there exists a (master) c.a.h. $\{3\}$-{\rm GDD} of type
$[g_1,g_2,\ldots,g_t]$, and suppose that an (ingredient) 
c.a.h. $\{3\}$-{\rm GDD} of type $h^{g_i/h} y^1$ exists for each $i\in[t]$.
Then there exists a c.a.h. $\{3\}$-GDD of type $h^{(\sum_{i=1}^t g_i)/h} y^1$.
\end{proposition}

\begin{proof}
Let $(X,\G,\A)$ be a c.a.h. $\{3\}$-GDD of type $[g_1,g_2,\ldots,g_t]$, 
with $\G=\{G_1,G_2$, $\ldots,G_t\}$.
Let $Y$ be a set of size $y$, disjoint from $X$. Let $(G_i\cup Y,\G_i,\B_i)$ be a c.a.h.
$\{3\}$-GDD of type $h^{|G_i|/h} y^1$, with $Y$ as the group of size $y$. Then the set system
\begin{equation*}
\fS = (X\cup Y,\cup_{i=1}^t \G_i,\A\cup(\cup_{i=1}^t \B_i))
\end{equation*}
is a $\{3\}$-GDD of type $h^{(\sum_{i=1}^t g_i)/h} y^1$. To show that $\fS$ is c.a.h., mimic
the proof of Proposition \ref{fillin}.
\end{proof}



\begin{figure}
\caption{Wilson's Fundamental Construction for GDDs.}
\label{Wilson}
\centering
\begin{tabular}{ l l}
\\
\hline
Input: & (master) GDD $\D=(X,\G,\A)$; \\
           & weight function $\omega:X\rightarrow\mathbb{Z}_{\geq 0}$; \\
           & (ingredient) $K$-GDD $\D_A=(X_A,\G_A,\B_A)$ of type
           $[\omega(a):a\in A]$, \\ 
           & for each block $A\in\A$, where \\
           & ~~~~~~~~~$X_A=\cup_{a\in A}\{\{a\} \times \{1,2,\ldots,\omega(a)\}\}$, and \\
           & ~~~~~~~~~$\G_A=\{\{a\}\times\{1,2,\ldots,\omega(a)\}:a\in A\}$. \\
Output: & $K$-GDD $\D^*=(X^*,\G^*,\A^*)$ of type $[\sum_{x\in G}\omega(x):G\in\G]$, where \\
              & ~~~~~~~~~$X^*=\cup_{x\in X}(\{x\}\times\{1,2,\ldots,\omega(x)\})$, \\
              & ~~~~~~~~~$\G^*=\{\cup_{x\in G}(\{x\}\times\{1,2,\ldots,\omega(x)\}) : G\in\G\}$, and \\
              & ~~~~~~~~~$\A^*=\cup_{A\in\A}\B_A$. \\
Notation: & $\D^*={\rm WFC}(\D,\omega,\{\D_A:A\in\A\})$. \\
Note: & By convention, for $x\in X$, $\{x\}\times\{1,2,\ldots,\omega(x)\}=\emptyset$ if $\omega(x)=0$.\\
\hline
\end{tabular}
\end{figure}

For Propositions \ref{fillin}, \ref{adjoin}, and \ref{break}
 to be useful, large classes of c.a.h. $\{3\}$-GDDs are needed.
These can be produced with
the next theorem, a direct analogue of
Wilson's Fundamental Construction for GDDs \cite{Wilson:1972a}
(shown in Fig. \ref{Wilson}), for c.a.h. GDDs.

\begin{theorem}[Fundamental Construction]
\label{WFC}
Let $\D=(X,\G,\A)$ be a (master) GDD,
and $\omega:X\rightarrow\bbZ_{\geq 0}$ be a weight function.
Suppose that for each $A\in\A$, there exists an (ingredient) c.a.h.
$K$-{\rm GDD} $\D_A$
of type $[ \omega(a):a\in A]$. Then there exists a 
c.a.h. $K$-{\rm GDD} $\D^*$
of type $[\sum_{x\in G}\omega(x): G\in\G]$. 
\end{theorem}

\begin{proof}
Apply Wilson's Fundamental Construction (Fig. \ref{Wilson}) to obtain a $K$-GDD $\D^*$
of type $[\sum_{x\in G}\omega(x): G\in\G]$. That this GDD is c.a.h. can be seen as follows.
Let $C_A$ be a 
c.a.h. cycle in the $K$-GDD $\D_A$, $A\in\A$. 
There exists an edge
in each of $C_A$ and $C_{A'}$ of the same color if $A\cap A'\not=\emptyset$, and
$C_A$ and $C_{A'}$ can be combined into one
c.a.h. cycle (with respect to colors
in $A\cap A'$) via Proposition \ref{C+C}.

The construction for a c.a.h.
cycle in $\D^*$ proceeds as follows. Start with the set of
cycles $\C=\{C_A:A\in\A\}$. As long as $\C$ contains more than one cycle, choose two
cycles in $\C$, each containing an edge of the same color, and combine them. This can always
be done unless the set $\C$ is reduced to a set of cycles, each containing edge colors that appear in no other cycles. However, this is impossible,
since every pair of points (which corresponds to colors of edges) in $\D^*$ appears in
some block, and hence some cycle of $\C$. The result is therefore a
c.a.h. cycle in $\D^*$.
\end{proof}

To seed the recursive constructions above, some small c.a.h. set systems are required.
These are given in the next subsection.

\subsection{Small Orders} \hfill

\begin{proposition}
\label{sts}
There exists a c.a.h. {\rm STS}$(n)$ for
$n\in\{3,7,9,13,15\}$.
\end{proposition}

\begin{proof}
The proposition is trivially true for $n=3$.
The required c.a.h. {\rm STS}$(n)$ for $n\in\{7,9,13,15\}$ are given
in Appendix \ref{smallsts}.
\end{proof}

\begin{proposition}
\label{mincover}
\ There exists an alternating hamiltonian minimum $(n,3,2)$-covering for
$n\in\{4,5,6,8,10,11,12,14,16,20\}$.
\end{proposition}

\begin{proof}
The required alternating hamiltonian minimum $(n,3,2)$-coverings are given
in Appendix \ref{smallcovering}.
\end{proof}

\begin{proposition}
\label{miscsmallgdd}
There exists an alternating hamiltonian $\{3\}$-{\rm GDD} of type $2^3$
and a c.a.h. $\{3\}$-{\rm GDD} of the following types: \\
\begin{tabular}{ccccccc}
{\rm (i)} $2^4$ & {\rm (ii)} $1^6 3^1$ & {\rm (iii)} $3^3$ & {\rm (iv)} $2^3 4^1$ &
{\rm (v)} $2^6$ & {\rm (vi)} $1^{12}3^1$ & {\rm (vii)} $5^3$ \\
\end{tabular}
\end{proposition}

\begin{proof}
For the alternating hamiltonian $\{3\}$-GDD $(X,\G,\A)$
of type $2^3$, take $X=[6]$, the groups to be $\{i,i+3\}$, $i\in[3]$, and
the blocks to be $\{2,1,3\}$, $\{3,4,5\}$, $\{5,1,6\}$, and $\{6,4,2\}$.

The other required c.a.h. $\{3\}$-GDDs are given in Appendix \ref{3GDDa}, noting the following.

For the c.a.h. $\{3\}$-GDD of type $2^4$, the groups are $\{i,i+4\}$, $i\in[4]$.

For the c.a.h. $\{3\}$-GDD of type $1^63^1$, the group of size three is $\{4,5,9\}$.

For the c.a.h. $\{3\}$-GDD of type $3^3$, the groups are $\{i,i+3,i+6\}$, $i\in[3]$.

For the c.a.h. $\{3\}$-GDD of type $2^34^1$, the groups are $\{1,5\}, \{2,6\}, \{3,7\}, \{4,8$, $9,10\}$.

For the c.a.h. $\{3\}$-GDD of type $2^6$, the groups are $\{i,i+6\}$, $i\in[6]$.

For the c.a.h. $\{3\}$-GDD of type $1^{12}3^1$, the group of size three is $\{3,8,14\}$.

For the c.a.h. $\{3\}$-GDD of type $5^3$, the groups are $\{i,i+3,i+6,i+9,i+12\}$, $i\in[3]$. 
\end{proof}

\begin{proposition}
\label{6^t}
There exists a c.a.h. $\{3\}$-{\rm GDD} of type $6^t$, for $t\in\{3$, $4$, $5$, $6$, $8\}$.
\end{proposition}

\begin{proof}
For $t\in\{3,4,6\}$, take a $\{3\}$-GDD of type $2^t$, which exists by Theorem \ref{CHR},
as master GDD and apply Theorem \ref{WFC} with weight function $\omega(\cdot)=3$.
The required ingredient c.a.h. $\{3\}$-GDD of type $3^3$ exists by 
Proposition \ref{miscsmallgdd}.

For $t\in\{5,8\}$, take a $\{4\}$-GDD of type $3^t$, which exists by Theorem \ref{BSH},
as master GDD and apply Theorem \ref{WFC} with weight function $\omega(\cdot)=2$.
The required ingredient c.a.h. $\{3\}$-GDD of type $2^4$ exists by
Proposition \ref{miscsmallgdd}.
\end{proof}

\begin{proposition}
\label{6^tu^1}
There exists a c.a.h. $\{3\}$-{\rm GDD} of type $6^tu^1$, for $n\in\{3$, 
$4$, $5$, $6$, $7$, $8$, $9$, $10$, $11$, $13$, $14$, $17$, $18$, $22\}$
and $u\in\{4,8\}$.
\end{proposition}

\begin{proof}
For $t\in\{5,6,10\}$ and $u\in\{4,8\}$, take the $\{4,7\}$-GDD of type $3^t u^1$ in
Appendix \ref{47GDD} as master GDD, and 
apply Wilson's Fundamental Construction (Fig. \ref{Wilson})
with weight function $\omega$ that assigns weight zero to the $u/2$ points
$3t+u/2+1,3t+u/2+2,\ldots,3t+u$ in the group
of size $u$, and weight two to each of the remaining points. Use as ingredient GDDs,
an alternating hamiltonian $\{3\}$-GDD of type $2^3$ and
c.a.h. $\{3\}$-GDDs of types $2^4$ and $2^6$, all of which exist by 
Proposition \ref{miscsmallgdd}. The result is a $\{3\}$-GDD $\D^*$ of type
$6^t u^1$. 
Let $\A'$ be the set of blocks, indicated in bold (and also in italics when $u=8$), 
of the master GDD. Then $\cup_{A\in\A'}\B_A$
(with the notations in Fig. \ref{Wilson}), together with the blocks of $\D^*$,
is a c.a.h. cycle in $\D^*$.

For $t\in\{3,4,7,8,11\}$ and $u=4$, take the $\{4\}$-GDD of type $3^{t+1}$ in
Appendix \ref{4GDD3^t} as master GDD, and apply Wilson's Fundamental Construction
(Fig. \ref{Wilson})
with weight function $\omega$ that assigns weight zero to the point $3(t+1)$, and weight
two to each of the remaining points. Use as ingredient GDDs, 
an alternating hamiltonian $\{3\}$-GDD of type $2^3$ and
a c.a.h. $\{3\}$-GDD of type $2^4$, both of which exist by 
Proposition \ref{miscsmallgdd}. The result is a $\{3\}$-GDD $\D^*$ of type $6^t 4^1$.
Let $\A'$ be the set of blocks, indicated in bold, of the master GDD. Then $\cup_{A\in\A'}\B_A$
(with the notations in Fig. 1), together with the blocks of $\D^*$,
is a c.a.h. cycle in $\D^*$.

For $t\in\{3,4,7,8,11\}$ and $u=8$, take a $\{4\}$-GDD of type $3^{t+1}$, which exists
by Theorem \ref{BSH}, as master GDD and apply Theorem \ref{WFC} with weight
function that assigns weight four to one point, and weight two to each of the remaining points.
The required ingredient c.a.h. $\{3\}$-GDDs of types $2^4$ and $2^3 4^1$ exist by
Proposition \ref{miscsmallgdd}. 

For the remaining $t$ and $u$, we break into three cases.

\begin{description}
\item[$t\in\{9,18\}$] Take a $\{3\}$-GDD of type $(t/3)^3$,
which exists by Theorem \ref{BSH}, as master GDD and apply Theorem \ref{WFC}
with weight function $\omega(\cdot)=6$ to obtain a c.a.h. $\{3\}$-GDD $\fS$ of type $(2t)^3$.
The required ingredient c.a.h. $\{3\}$-GDD of type $6^3$ exists by Proposition \ref{6^t}.
Now adjoin $u$ points and break up the groups of $\fS$ (Proposition \ref{break}) using a c.a.h.
$\{3\}$-GDD of type $6^{t/3} u^1$, whose existence has been established above,
to obtain a c.a.h. $\{3\}$-GDD of type $6^t u^1$.

\item[$t\in\{13,14\}$] Take a $\{3\}$-GDD of type $6^3 (2(t-9))^1$, which exists by Theorem
\ref{BSH}, as master GDD and apply Theorem \ref{WFC} with weight function $\omega(\cdot)=3$,
to obtain a c.a.h. $\{3\}$-GDD $\fS$ of type $18^3 (6(t-9))^1$. The required ingredient c.a.h.
$\{3\}$-GDD of type $3^3$ exists by Proposition \ref{miscsmallgdd}.
Now adjoin $u$ points and break up the groups of $\fS$ (Proposition \ref{break})
using c.a.h. $\{3\}$-GDDs of type $6^s u^1$, $s\in\{3,t-9\}$, whose existence has
been established above, to obtain a c.a.h. $\{3\}$-GDD of type $6^t u^1$.

\item[$t\in\{17,22\}$]  
Take a $\{3\}$-GDD of type $((t-2)/5)^4 ((t+8)/5)^1$, which exists by Theorem
\ref{BSH}, as master GDD and apply Theorem \ref{WFC} with weight function $\omega(\cdot)=6$,
to obtain a c.a.h. $\{3\}$-GDD $\fS$ of type $(6(t-2)/5)^4 (6(t+8)/5)^1$. The required ingredient c.a.h.
$\{3\}$-GDD of type $6^3$ exists by Proposition \ref{miscsmallgdd}.
Now adjoin $u$ points and break up the groups of $\fS$ (Proposition \ref{break})
using c.a.h. $\{3\}$-GDDs of type $6^s u^1$, $s\in\{(t-2)/5,(t+8)/5\}$, whose existence has
been established above, to obtain a c.a.h. $\{3\}$-GDD of type $6^t u^1$.
\end{description}
\end{proof}

\subsection{Piecing Things Together} \hfill

\begin{proposition}
\label{all6^t}
There exists a c.a.h. $\{3\}$-{\rm GDD} of type $6^t$, for all $t\geq 3$.
\end{proposition}

\begin{proof}
Take a PBD$(t,\{3,4,5,6,8\})$, which exists by Theorem \ref{GMP}, as master GDD, and
apply Theorem \ref{WFC} with weight function $\omega(\cdot)=6$. The required
ingredient c.a.h. $\{3\}$-GDDs of type $6^s$, $s\in\{3,4,5,6,8\}$, exist by Proposition \ref{6^t}.
\end{proof}

\begin{proposition}
\label{all6^tu^1}
There exists a c.a.h. $\{3\}$-{\rm GDD} of type $6^t u^1$, for all $t\geq 3$ and $u\in\{4,8\}$. 
\end{proposition}

\begin{proof}
For $t\in\{3,4,5,6,7,8,9,10,11,13,14,17,18,22\}$ and $u\in\{4,8\}$, the required GDDs exist
by Proposition \ref{6^tu^1}. For other values of $t$, take a
PBD$(t+1,\{4,5,6,7\})$, which exists by Theorem \ref{Lenz}, as master GDD, and
apply Theorem \ref{WFC} with weight function $\omega$ that assigns weight six to each of
$t$ points and weight $u$ to the remaining point. The required ingredient c.a.h. $\{3\}$-GDDs of
types $6^s$ and $6^{s-1} u^1$, $s\in\{4,5,6,7\}$, exist by either
Proposition \ref{all6^t} or Proposition \ref{6^tu^1}.
\end{proof}

We are now ready to state the main result of this section.

\begin{theorem}
\  \ There exists an alternating hamiltonian minimum $(n,3,2)$-covering, for all $n\geq 3$.
\end{theorem}

\begin{proof}
The existence of alternating hamiltonian minimum $(n,3,2)$-coverings for $n\in\{3,4,\ldots,16\}\cup\{20\}$
has been established by either Proposition \ref{sts} or Proposition \ref{mincover}.

A c.a.h. minimum $(17,3,2)$-covering can be obtained by adjoining two points and filling
in the groups of a c.a.h. $\{3\}$-GDD of type $5^3$ with a c.a.h. STS$(7)$, which
exist by Propositions \ref{sts} and \ref{miscsmallgdd}.

For $n\equiv 0\pmod{6}$, $n\geq 18$, take a c.a.h. $\{3\}$-GDD of type $6^{n/6}$,
which exists by Proposition \ref{all6^t}, and fill in groups with a c.a.h. minimum
$(6,3,2)$-covering, which exists by Proposition \ref{mincover}.

For $n\equiv 1\pmod{6}$, $n\geq 19$, take a c.a.h. $\{3\}$-GDD of type $6^{(n-1)/6}$, which
exists by Proposition \ref{all6^t}, adjoin one point and fill in groups 
with a c.a.h. STS$(7)$, which exists by Proposition \ref{sts}.

For $n\equiv 2$ or $4\pmod{6}$, $n\geq 22$, take a c.a.h. $\{3\}$-GDD
of type $6^{(n-8)/6} 8^1$ or of type $6^{(n-4)/6} 4^1$ (which exists by
Proposition \ref{all6^tu^1}), respectively, and fill in groups with alternating hamiltonian
minimum $(m,3,2)$-coverings, $m\in\{4,6,8\}$, which exist by Proposition \ref{mincover}.

For $n\equiv 3\pmod{6}$, $n\geq 21$, take a c.a.h. $\{3\}$-GDD of type $6^{(n-3)/6}$,
adjoin three points and fill in the groups with c.a.h. $\{3\}$-GDD of type $1^6 3^1$ and
c.a.h. STS$(9)$, which exist by Propositions \ref{sts} and \ref{miscsmallgdd}.

For $n\equiv 5\pmod{6}$, $n\geq 23$, take a c.a.h. $\{3\}$-GDD of type $6^{(n-5)/6} 4^1$,
adjoin one point and fill in the groups with c.a.h. STS$(7)$
and an alternating hamiltonian minimum $(5,3,2)$-covering, 
which exist by Propositions \ref{sts} and \ref{mincover}.
\end{proof}

\begin{corollary}
\label{Ucycle}
There exists a minimum $(n,3,2)$-covering which possesses a $2$-shift universal cycle,
for all $n\geq 3$.
\end{corollary}

\section{Application to 2-Radius Sequences}

Jaromczyk and Lonc \cite{JaromczykLonc:2004} studied sequences
$(a_1,a_2,\ldots,a_m)\in[n]^m$
with the property that for every distinct $x,y\in[n]$,
there exist $a_i$ and $a_j$ such that $\{a_i,a_j\}=\{x,y\}$ and $|i-j| \leq k$. 
These sequences are known as {\em $k$-radius sequences of order $n$}. 
For given $n$ and $k$,
the objective is to find a shortest (optimal) $k$-radius sequence of order $n$. The length of such
a sequence is denoted $f_k(n)$. 

Motivation for studying short $k$-radius sequences comes from computation where a
certain two argument function must be evaluated for all pairs of $n$ large objects
$\O_1,\O_2,\ldots,\O_n$, which are
too large to be all held in fast primary memory 
(such as internal memory in the I/O model of Aggarwal and Vitter \cite{AggarwalVitter:1988},
and cache memory in the cache model of Sen {\em et al.} \cite{Senetal:2002}) at once.
Hence, instead of the simple two-loop algorithm
that iterates through all the pairs of objects, a schedule to determine which objects are to
be fetched into memory and which are to be replaced, is needed. This schedule must 
ensure that for all pairs of objects $x$ and $y$, there is some point in time where $x$
and $y$ are both in memory. If the memory can only hold $k+1$ objects at any one time,
such a schedule corresponds to a $k$-radius sequence $(a_1,a_2,\ldots,a_m)$
of order $n$, assuming a 
first-in-first-out (FIFO) object replacement strategy \cite{JaromczykLonc:2004}:
initially (at time step one)
the objects $\O_{a_1},\O_{a_{2}},\ldots,\O_{a_{k+1}}$ reside in primary memory
and at step time step $t\geq 2$, the object $\O_{a_{k+t}}$ is fetched to replace
object $\O_{a_{t-1}}$.
For
reasons of efficiency, short schedules are desirable. 

Extending beyond cache algorithms, $k$-radius sequences are also applicable when
the large objects reside in remote servers and local storage is not large enough to store
all the objects for computations required to run over all
possible pairs of objects. For bandwidth efficiency, we would like to
minimize the fetching of the large objects over the network into local storage. 
An example of this scenario is the restoration of the order of images in a sequence of MRI slices,
when the MRI images (which are typically quite large) reside in remote databases
\cite{GilkersonJaromczyk:2002}.

The function $f_1(n)$ has been completely determined by Ghosh \cite{Ghosh:1975}
in the context of database theory.

\begin{theorem}[Ghosh \cite{Ghosh:1975}]
\begin{equation*}
f_1(n) =\begin{cases}
\binom{n}{2}+1,&\text{if $n$ is odd;} \\
\binom{n}{2}+\frac{1}{2}n,&\text{if $n$ is even.}
\end{cases}
\end{equation*}
\end{theorem}

The function $f_2(n)$ is recently investigated by 
Jaromczyk and Lonc \cite{JaromczykLonc:2004}, who established the bounds below.

\begin{theorem}[Jaromczyk and Lonc \cite{JaromczykLonc:2004}]
\label{JL}
\begin{equation*}
L(n) \leq f_2(n)\leq \frac{1}{2}\binom{n}{2} + \frac{10n^2}{\log_2 n} + 2n^{1.64},
\end{equation*}
where
\begin{equation*}
L(n) = \begin{cases}
\frac{1}{2}\binom{n}{2}+\frac{1}{4}n+1,&\text{if $n\equiv 0\pmod{4}$;} \\
\frac{1}{2}\binom{n}{2}+2,&\text{if $n\equiv 1\pmod{4}$;} \\
\frac{1}{2}\binom{n}{2}+\frac{3}{4}n,&\text{if $n\equiv 2\pmod{4}$;} \\
\frac{1}{2}\binom{n}{2}+\frac{1}{2}n,&\text{if $n\equiv 3\pmod{4}$.} \\
\end{cases}
\end{equation*}
\end{theorem}

The upper bound in Theorem \ref{JL} comes from a number-theoretic
construction \cite[Section 2]{JaromczykLonc:2004}, which we refer to as the
Jaromczyk-Lonc construction. Exact values of $f_2(n)$ are known previously only for $n\leq 7$ 
\cite{JaromczykLonc:2004,Gilkersonetal:unpublished}. 
Next, 
$f_2(n)$ is determined for $n\in\{8,10,11,14,15\}$ and better upper bounds on $f_2(n)$,
for $n\in\{9,12,13,16,17,18\}$, are given.

\subsection{Some New 2-Radius Sequences}
\label{new2radius}

The value of $f_2(n)$ meets the lower bound in Theorem \ref{JL} for $n\in\{8,10,11,14,15\}$.
The new optimal $2$-radius sequences proving this are given in Table \ref{newoptimal}.

\begin{table}
\caption{New optimal $2$-radius sequences of order $n$}
\label{newoptimal}
\footnotesize
\centering
\begin{tabular}{|c|c|l|}
\hline
$n$ & length & sequence \\
\hline
8 & 17 & $(3,6,2,7,8,5,6,4,1,8,7,3,4,2,5,1,3)$ \\
\hline
10 & 30 & $(2,6,8,10,9,1,5,4,7,10,2,6,4,5,9,2,1,3,10,5,8,7,9,3,6,7,1,8,4,3)$ \\
\hline
11 & 33 & $(10,9,7,11,2,3,1,10,11,8,6,3,1,7,4,11,3,5,9,6,4,10,2,5,6,4,7,8,5,9,1,2,8)$ \\
\hline
12 & 37 & $(12,6,7,11,10,9,3,6,5,2,9,1,8,5,7,3,12,1,10,6,4,8,3,11,2,1,7,4,9,12,8,2,$ \\
& & $10,4,5,11,12)$ \\
\hline
14 & 56 & $(13,9,8,12,2,5,3,11,13,14,4,9,10,1,6,14,8,3,7,10,11,8,14,5,7,13,12,6,4,8,$ \\
& & $1,3,6,2,11,9,12,3,10,4,13,2,1,7,11,4,1,5,12,10,14,2,5,9,6,7)$ \\
\hline
15 & 60 & $(10,9,7,1,11,14,9,15,12,2,8,6,13,7,5,3,15,6,11,12,5,14,2,3,10,12,1,13,3,$ \\
& & $4,11,15,8,10,5,6,9,3,8,5,1,4,2,6,14,1,13,15,7,8,14,4,10,11,13,2,9,7,4,$ \\
& & $12)$ \\
\hline
16 & 65 & $(7,8,9, 6,10,4,15,13,1,6,12,3,15,8,16,12,10,5,2,15,6,11,16,2,13,9,12,16,$ \\
& & $14,4,3,5,11,8,10,13,7,3,15,9,14,11,13,7,5,16,1,3,10,2,14,8,4,1,2,7,12,4,$ \\
& & $11,9,1,5,14,6,7) $ \\
\hline
18 & 90 & $(2,9,7,1,17,5,12,13,4,17,9,18,6,1,15,8,13,16,6,4,7,1,16,10,18,17,6,14,11,$ \\
& & $7,12,10,14,15,13,11,16,3,15,12,18,4,3,10,6,1,12,14,8,5,18,7,13,8,9,10,3,$ \\
& & $13,1,2,11,8,17,3,5,2,6,8,11,4,14,2,18,10,11,5,9,16,15,2,17,12,16,9,14,3,$ \\
& & $7,15,4,5)$ \\
\hline
\end{tabular}
\end{table}

When $n\in\{9,12,13,16,17,18\}$, we improve the on the shortest 2-radius sequence
of order $n$ currently known
\cite{JaromczykLonc:2004}. The new $2$-radius sequences are given in Table \ref{newshort}.
These sequences are not known to be optimal.

\begin{table}
\caption{New short $2$-radius sequences of order $n$}
\label{newshort}
\footnotesize
\centering
\begin{tabular}{|c|c|l|}
\hline
$n$ & length & sequence \\
\hline
9 & 21 & $(9,3,2,7,8,4,9,3,1,8,5,9,7,6,1,2,4,5,6,3,8)$ \\
\hline
13 & 42 & $(4,9,11,5,7,4,8,2,11,3,1,4,12,10,11,13,6,4,9,1,8,10,6,3,12,7,9,10,2,5,6,$ \\
& & $1,7,2,13,12,5,8,3,13,9,1)$ \\
\hline 
17 & 73 & $(1,2,3,4,5,1,6,7,3,8,9,1,10,11,3,12,13,1,14,15,3,16,17,2,5,7,9,4,6,8,2,11,$ \\
& & $13,4,10,12,5,8,13,15,6,10,14,2,9,15,12,7,14,11,5,15,17,4,14,16,8,10,17,7,$\\
& & $13,16,9,11,17,6,12,16,2,5,1,17,16)$ \\
\hline
\end{tabular}
\end{table}

\begin{table}
\caption{State of knowledge of $f_2(n)$, $2\leq n\leq 18$}
\label{f2}
\footnotesize
\centering
\setlength\tabcolsep{4pt}
\begin{tabular}{|c|ccccccccccccccccc|}
\hline
$n$ & $2$ & $3$ & $4$ & $5$ & $6$ & $7$ & $8$ & $9$ & $10$ & $11$ & $12$ & $13$ & $14$ & $15$ & $16$ & $17$ & $18$ \\
\hline
$f_2(n)$ & $2$ & $3$ & $5$ & $7$ & $12$ & $14$ & {\bf 17} & {\bf 20--21} & {\bf 30} & {\bf 33} & {\bf 37} & {\bf 41--42} & {\bf 56} & {\bf 60} & {\bf 65} & {\bf 70--73} & {\bf 90}  \\
\hline
\end{tabular}
\end{table}

The new 2-radius sequences obtained above are all found through
computer search. 
The state of knowledge for small values of $n$
is provided in Table \ref{f2}, where an entry with a single number gives the
exact value of $f_2(n)$, and an entry of the form ``$a$--$b$'' means that the
corresponding value of $f_2(n)$ lies between $a$ and $b$ (inclusive). A bold entry
indicates new results obtained in this paper, while the other entries are from
\cite{JaromczykLonc:2004,Gilkersonetal:unpublished}.

We now give details of our search procedure. The framework used is {\em hillclimbing}.
To construct a 2-radius sequence of order $n$ and length $m$, we start with a
random sequence in
$[n]^m$ and modify it iteratively to get ``closer'' and ``closer'' to a
2-radius sequence, until either we end up with a 2-radius sequence, or we get stuck.
We measure the ``closeness'' of a sequence
$S=(a_1,a_2,\ldots,a_m)\in[n]^m$
to a 2-radius sequence by its {\em defect}, denoted ${\rm def}(S)$, defined as
the number of pairs $\{x,y\}\subseteq[n]$
for which there do not exist $a_i,a_j$ such that $\{a_i,a_j\}=\{x,y\}$ and $|i-j|\leq 2$.
A sequence is therefore a 2-radius sequence if and only if it has zero defect.
At each step, modification of a sequence $(a_1,a_2,\ldots,a_m)$
proceeds as follows. We select three distinct
positions $i,j,k\in[m]$ and replace the values $a_i$, $a_j$, and $a_k$ by
elements of $[n]$, drawn uniformly at random. If this modification does not result
in a new sequence of higher defect, we accept the modification. Otherwise, we reject
the modification. The procedure is terminated when we have a sequence of zero defect,
and is restarted after a prespecified period of time without finding a defect-reducing
modification.
This hillclimbing procedure is described more formally in pseudocode in Algorithm \ref{hc}.

\begin{center}
\begin{minipage}{\textwidth }
\begin{algorithm}[H]
\KwIn{$n$, $m$}
\KwOut{2-radius sequence $S=(a_1,a_2,\ldots,a_m)$ of order $n$}

$S$ = random sequence $(a_1,a_2,\ldots,a_m)\in[n]^m$ \;
\While{${\rm def}(S)>0$}{
	$\{i,j,k\}$ = random 3-subset of $[m]$ \;
	$(r_i,r_j,r_k)$ = random element of $[n]^3$ \;
	$S'$ = $(b_1,b_2,\ldots,b_m)$, where $b_h = \begin{cases}
	a_h,&\text{if $h\in[m]\setminus\{i,j,k\}$} \\
	r_h,&\text{if $h\in\{i,j,k\}$}
	\end{cases}$ \;
	\If{${\rm def}(S') \leq {\rm def}(S)$}{
		$S$ = $(a_1,a_2,\ldots,a_m)$ =  $S'$ \;
	}
}
\caption{Hillclimbing procedure for constructing 2-radius sequences}
\label{hc}
\end{algorithm}
\end{minipage}
\end{center}

All the 2-radius sequences
in Tables \ref{newoptimal} and \ref{newshort} are obtained using Algorithm \ref{hc}.

\subsection{Improvements to Theoretically Provable Bounds}

Although the bounds in Theorem \ref{JL} are asymptotically tight, proving
$f_2(n)=(1+o(1)) n^2/4$, the upper bound
is rather weak. Gilkerson {\em et al.} \cite{Gilkersonetal:unpublished}
proved the following upper bound, which although not asymptotically tight, improves on
the upper bound in Theorem \ref{JL} for $n \leq 1.329\times 10^{36}$.

\begin{theorem}[Gilkerson {\em et al.} \cite{Gilkersonetal:unpublished}]
\label{G}
Let $m^2|n$ such that there exists either a {\rm PBD}$(m^2,\{m\})$ (affine plane of order $m$)
or a {\rm PBD}$(m^2-m+1,\{m\})$ (projective plane of order $m-1$). 
Then $f_2(n) \leq n^2/3+n$.
\end{theorem}

However, at present, PBD$(m^2,\{m\})$ and PBD$(m^2-m+1,\{m\})$ are only known to exist
only when $m$ or $m-1$ is a prime power, respectively
(see, for example, \cite{Storme:2007}). Hence, the applicability of
Theorem \ref{G} is limited. The results in \S\ref{AH} imply the following stronger result.

\begin{theorem}
\label{CLTZ}
For all $n\geq 3$, $f_2(n) \leq 2C(n,3,2)+1$.
\end{theorem}

\begin{proof}
Follows directly from Corollary \ref{Ucycle} 
and the observation that if $(u_0,u_1,\ldots$, $u_{2m-1})$ is
a 2-shift universal cycle for an $(n,3,2)$-covering, then
$(u_0,u_1,\ldots,u_{2m-1}$, $u_0)$ is a 2-radius sequence of order $n$ and
length $2m+1$.
\end{proof}

The upper bound in Theorem \ref{CLTZ} is strictly better than the upper bound in
Theorem \ref{G} for all $n\geq 3$, since 
\begin{equation*}
\left(\frac{n^2}{3}+n\right) - (2C(n,3,2)+1) = \begin{cases}
\frac{4n}{3}-1,&\text{if $n\equiv 1$ or $3\pmod{6}$;} \\
n-\frac{5}{3},&\text{if $n\equiv 2$ or $4\pmod{6}$;} \\
\frac{4n-7}{3},&\text{if $n\equiv 5\pmod{6}$;} \\
n-1,&\text{if $n\equiv 0\pmod{6}$.}
\end{cases}
\end{equation*}
It is also better than the upper bound in Theorem \ref{JL}
for all $n\leq 1.329 \times 10^{36}$. 

\subsection{Actual Performance of Constructions}

The upper bound in Theorem \ref{JL} is what is theoretically provable of
the Jaromczyk-Lonc construction of 2-radius sequences. However, its actual
performance is much better. We provide in Table \ref{actual} a comparison of the
lengths of 2-radius sequences actually produced by the Jaromczyk-Lonc construction
and that obtained through our bound in Theorem \ref{CLTZ}. In Table \ref{actual},
\begin{align*}
{\rm len}_{\rm JL~theory} &= \text{theoretically provable length of 2-radius sequence of order $n$} \\
& \text{ \phantom{=} produced by the Jaromczyk-Lonc construction}, \\
{\rm len}_{\rm JL~actual} &= \text{actual length of 2-radius sequence of order $n$ produced 
by the} \\
& \text{ \phantom{=} Jaromczyk-Lonc construction}, \\
{\rm len}_{\rm this} &= \text{actual (and also theoretically provable) length of 2-radius sequence} \\
& \text{ \phantom{=} of order $n$ produced by our construction in this paper}.
\end{align*}
Entries in bold denote orders for which our construction outperforms the actual performance
of the Jaromczyk-Lonc construction.

\begin{table}[h!]
\caption{Performance comparison for constructions of 2-radius sequences
of order $n$, $9\leq n\leq 44$}
\label{actual}
\footnotesize
\centering
\setlength\tabcolsep{4pt}
\begin{tabular}{|l|cccccccccccc|}
\hline
$n$ & $9$ & $10$ & $11$ & $12$ & $13$ & $14$ & $15$ & $16$ & $17$ & $18$ & $19$ & $20$ \\
\hline
${\rm len}_{\rm JL~theory}$ & 346 & 410 & 479 & 552 & 629 & 711 & 798 & 888 & 983 & 1082 & 1185 & 1292 \\
\hline
${\rm len}_{\rm JL~actual}$ & 37 & 49 & 39 & 53 & 45 & 62 & 80 & 99 & 76 & 98 & 105 & 129 \\
\hline
${\rm len}_{\rm this}$ & {\bf 25} & {\bf 35} & 39 & {\bf 49} & 53 & 67 & {\bf 71} & {\bf 87} & 93 & 109 & 115 & 135 \\
\hline 
\multicolumn{13}{c}{\phantom{a}} \\
\hline
$n$ & $21$ & $22$ & $23$ & $24$ & $25$ & $26$ & $27$ & $28$ & $29$ & $30$ & $31$ & $32$ \\
\hline
${\rm len}_{\rm JL~theory}$ & 1403 & 1518 & 1638 & 1761 & 1888 & 2019 & 2153 & 2292 & 2434 & 2580 & 2730 & 2884  \\
\hline
${\rm len}_{\rm JL~actual}$ & 158 & 185 & 150 & 179 & 170 & 202 & 256 & 290 & 217 & 254 & 297 & 336 \\
\hline
${\rm len}_{\rm this}$ & {\bf 141} & {\bf 163} & 171 & 193 & 201 & 227 & {\bf 235} & {\bf 263} & 273 & 301 & 311 & 343 \\
\hline
\multicolumn{13}{c}{\phantom{a}} \\
\hline
$n$ & $33$ & $34$ & $35$ & $36$ & $37$ & $38$ & $39$ & $40$ & $41$ & $42$ & $43$ & $44$ \\
\hline
${\rm len}_{\rm JL~theory}$ & 3041 & 3202 & 3366 & 3535 & 3707 & 3882 & 4061 & 4244 & 4430 & 4620 & 4813 & 5010 \\
\hline
${\rm len}_{\rm JL~actual}$ & 382 & 424 & 361 & 405 & 351 & 398 & 446 & 495 & 430 & 482 & 540 & 594 \\
\hline
${\rm len}_{\rm this}$ & {\bf 353} & {\bf 387} & 399 & 433 & 445 & 483 & 495 & 535 & 549 & 589 & 603 & 647 \\
\hline
\end{tabular}
\end{table}

\subsection{Recent Advancements}

Blackburn and Mckee \cite{BlackburnMckee:2010} have very recently revealed unexpected
connections between $k$-radius sequences, lattice tilings, and logarithms in $\bbZ_k$.
In particular, they obtained an upper bound on $f_2(n)$ that improves on
Theorem \ref{JL}.

\section{Conclusion}

A new ordering on the blocks of set systems, called shift universal cycles, is introduced.
Minimum $(n,3,2)$-coverings that admit $2$-shift universal cycles are shown to
exist for all $n$. These minimum $(n,3,2)$-coverings are used to construct short
$2$-radius sequences, which have applications in cache algorithms.

\appendix 

\section{\boldmath Some c.a.h. Set Systems}

A hamiltonian cycle of length $m$
in the block intersection graph of a set system or GDD, $\fS$, is specified
as a sequence of blocks in $\fS$: $A_0A_1\cdots A_{m-1}$. The edges of the
hamiltonian cycle is taken to be $\{A_i,A_{i+1}\}$, $i\in\bbZ_m$.

In each case, the set of points is taken to be $[n]$.

\subsection{\boldmath Some small c.a.h. STS$(n)$} \hfill
\label{smallsts}

{\footnotesize
\begin{longtable}{|c|l|}
\hline
$n$ & c.a.h. cycle in STS$(n)$ \\
\hline
$7$ & $\{5,3,1\}$ $\{1,7,4\}$ $\{4,2,3\}$ $\{3,6,7\}$ $\{7,5,2\}$  $\{2,1,6\}$ 
$\{6,4,5\}$  \\
\hline
$9$ & $\{3,6,9\}$ $\{9,5,1\}$ $\{1,6,8\}$ $\{8,7,9\}$ $\{9,4,2\}$ $\{ 2,8,5\}$ $\{5,3,7\}$
$\{7,4,1\}$ $\{1,3,2\}$ $\{2,7,6\}$ \\
& $\{6,5,4\}$ $\{4,8,3\}$ \\
\hline
$13$ & $\{5,10,3\}$ $\{3,13,12\}$ $\{12,7, 5\}$ $\{5,8,4\}$ $\{4,11,6\}$ $\{6,2,3\}$ $\{3,8,1\}$
$\{1,7,9\}$ $\{9,6,5\}$ \\
& $\{5,1,2\}$ $\{2,10,8\}$ $\{8,12,9\}$  $\{9,2,4\}$ $\{4,12,10\}$ $\{10,6,7\}$ $\{7,8,11\}$
$\{11,3,9\}$ $\{9,13,10\}$ \\
& $\{10,1,11\}$  $\{11,12, 2\}$ $\{2,13,7\}$ $\{7,3,4\}$ $\{4,13,1\}$ 
$\{1,12,6\}$ $\{6,8,13\}$ $\{13,11,5\}$ \\
\hline
$15$ & $\{7,11,1\}$ $\{1,10,5\}$ $\{5,6,15\}$ $\{15,8,11\}$ $\{11,2, 4\}$  $\{4,13,8\}$
$\{8,14,1\}$ $\{1,9,12\}$ $\{12,6,11\}$ \\
& $\{11,5,13\}$ $\{13,14,12\}$ $\{12,4,15\}$  $\{15,1,2\}$ $\{2,3,12\}$
$\{12,7,10\}$ $\{ 10,11,9\}$ $\{9,13,6\}$ \\ 
& $\{6,7,8\}$ $\{8,10,2\}$  $\{2,14,6\}$ $\{6,4,1\}$
$\{1,13,3\}$ $\{3,11,14\}$ $\{14,15, 9\}$ $\{9,7,4\}$ $\{4,10,14\}$ \\
& $\{14,5,7\}$ $\{7,3,15\}$
$\{15,13,10\}$ $\{10,6,3\}$ $\{3,4,5\}$ $\{5,12,8\}$ $\{8,3,9\}$  $\{9,5,2\}$ $\{2,13,7\}$ \\
\hline
\end{longtable}
}

\subsection{\boldmath Some small alternating hamiltonian minimum $(n,3,2)$-coverings} \hfill
\label{smallcovering}

The alternating hamiltonian minimum $(n,3,2)$-coverings given below
for $n\in\{6,8,10,11,12,14,16,20\}$ are in fact c.a.h.

{\footnotesize
\begin{longtable}{|c|l|}
\hline
$n$ & alternating hamiltonian cycle in minimum $(n,3,2)$-coverings \\
\hline
$4$ & $\{3,1,2 \}$ $\{2,4,1\}$ $\{1,4,3\}$ \\
\hline
$5$ & $\{ 1,2,3\}$ $\{ 3,4,5\}$ $\{ 5,2,4\}$ $\{4,5,1\}$ \\
\hline
$6$ & $\{6,3,1 \}$ $\{1,4,2 \}$ $\{2,5,3 \}$ $\{3,6,4 \}$ $\{4,1,5 \}$  $\{5,2,6 \}$ \\
\hline
$8$ &  $\{3,8,1 \}$ $\{ 1,4,2\}$ $\{ 2,5,3 \}$ $\{3,6,4 \}$ $\{4,8,5  \}$  $\{  5,1,6\}$ $\{6,2,8 \}$
 $\{ 8,1,7 \}$ $\{  7,4,5 \}$ $\{5,6,7  \}$  \\
 & $\{7,2,3  \}$ \\
 \hline
$10$ & $\{7,9,8\}$
    $\{8,4,3\}$
    $\{3,6,9\}$
    $\{9,10,8\}$
    $\{8,6,1\}$
    $\{1,10,2\}$
    $\{2,8,5\}$
    $\{5,1,9 \}$
    $\{9,2,4\}$ 
    $\{4,1,7\}$ \\
    &
    $\{7,8,10\}$
    $\{10,4,3\}$
    $\{3,1,2\}$
    $\{2,7,6\}$
    $\{6,10,5\}$
    $\{5,3,7\}$ \\
\hline
$11$ & $\{8,11,3\}$
    $\{3,6,10\}$
    $\{10,4, 9\}$
    $\{9,8,1\}$
    $\{1,3, 2\}$
    $\{ 2, 7, 10 \}$
    $\{ 10,1,11 \}$
    $\{11,7, 5 \}$
    $\{ 5, 9,6\}$ \\
&
    $\{ 6,11,4\}$
    $\{ 4,1, 2 \}$
    $\{ 2, 6, 8 \}$
    $\{ 8,10,5\}$
    $\{ 5,2,1 \}$
    $\{ 1, 6, 7 \}$
    $\{ 7,9,3 \}$
    $\{ 3, 5,4\}$
    $\{ 4, 7, 8 \}$ \\
\hline
$12$ & $\{11,6,4 \}$
    $\{ 4, 5, 12 \}$
    $\{ 12,3,1 \}$
    $\{ 1, 8, 9 \}$
    $\{ 9,11,2 \}$
    $\{ 2, 12,3\}$
    $\{ 3, 11,8 \}$
    $\{ 8, 12,9 \}$
    $\{ 9,4,10 \}$ \\
&
    $\{ 10,7,2\}$
    $\{ 2,1,5 \}$
    $\{ 5, 9,6 \}$
    $\{ 6, 12,7 \}$
    $\{ 7,5,11 \}$
    $\{ 11,1, 10 \}$
    $\{ 10,5, 8\}$
    $\{ 8,6,2 \}$
    $\{ 2,4,1 \}$ \\
&
    $\{ 1, 6, 7 \}$ 
    $\{ 7,9,3\}$
    $\{ 3, 6, 10 \}$
    $\{ 10, 12,11 \}$ \\
\hline
$14$ & $\{ 9, 10, 13 \}$
    $\{ 13,12, 14 \}$
    $\{ 14,9, 10 \}$
    $\{10, 3, 5\}$
    $\{ 5,8,4 \}$
    $\{ 4,13,1\}$
    $\{ 1, 8,3 \}$
    $\{ 3,6,2\}$ 
    $\{ 2, 7, 13 \}$ \\
    &
    $\{ 13,5, 11 \}$
    $\{ 11,6,4 \}$
    $\{ 4,3, 14 \}$
    $\{ 14,1, 2 \}$
    $\{ 2, 10,8 \}$
    $\{ 8, 12,9 \}$
    $\{ 9,6,5 \}$
    $\{ 5,2,1 \}$ 
    $\{ 1, 10, 11 \}$ \\
    &
    $\{ 11, 14,12\}$
    $\{ 12,4, 10 \}$
    $\{ 10,6, 7 \}$
    $\{ 7, 8, 11 \}$
    $\{ 11,2, 12 \}$
    $\{12, 3,13 \}$
    $\{ 13,6, 8 \}$ 
    $\{ 8,7, 14 \}$ \\
    &
    $\{ 14,5, 6 \}$
    $\{6,12, 1 \}$
    $\{ 1, 7, 9 \}$
    $\{ 9,2, 4 \}$
    $\{4,7, 3\}$
    $\{ 3, 11,9 \}$ \\
\hline
$16$ & $\{ 7,6, 8 \}$
    $\{ 8,3, 9 \}$
    $\{ 9,7,4 \}$
    $\{ 4, 13,8 \}$
    $\{ 8,14,1 \}$
    $\{ 1, 9, 12 \}$
    $\{ 12,11,6 \}$
    $\{ 6,14,2 \}$
    $\{ 2, 13,7 \}$ \\
&
    $\{ 7, 10, 12 \}$
    $\{ 12, 14,13 \}$
    $\{ 13,11,5 \}$
    $\{ 5,10,1 \}$
    $\{ 1, 4, 6 \}$
    $\{ 6,5, 16 \}$
    $\{16, 1, 2 \}$
    $\{ 2, 8, 10 \}$ \\
&
    $\{ 10,9, 11 \}$
    $\{ 11,7,1 \}$
    $\{ 1, 15,2\}$
    $\{ 2, 3, 12 \}$
    $\{ 12,4, 15 \}$
    $\{ 15,5, 6\}$
    $\{ 6, 13,9 \}$
    $\{ 9, 14, 15 \}$ \\
&
    $\{ 15,11,8 \}$
    $\{ 8,7, 16 \}$
    $\{ 16,9, 10 \}$
    $\{ 10,4, 14 \}$
    $\{ 14,16,13\}$
    $\{ 13,1, 3\}$
    $\{ 3, 14,11 \}$ 
    $\{ 11, 16,12 \}$ \\
    &
    $\{ 12,8,5\}$
    $\{ 5, 7, 14 \}$
    $\{ 14, 15, 16 \}$
    $\{ 16,3, 4 \}$
    $\{ 4,11,2 \}$
    $\{ 2, 9,5 \}$
    $\{ 5,4,3 \}$ 
    $\{ 3, 6, 10 \}$ \\
    &
    $\{ 10, 13, 15 \}$ 
    $\{ 15,3, 7 \}$ \\
\hline
$20$ & $\{ 12,5,3 \}$
    $\{ 3, 14,8\}$
    $\{ 8, 13, 19 \}$
    $\{ 19,11,5\}$
    $\{5,8, 4 \}$
    $\{ 4,2, 11 \}$
    $\{ 11, 12, 15 \}$
    $\{ 15,14, 18 \}$ \\
&
    $\{ 18,1, 8 \}$
    $\{ 8, 10, 17 \}$
    $\{ 17,3, 9 \}$
    $\{ 9,19,2\}$
    $\{ 2, 3, 6 \}$
    $\{ 6, 8, 15 \}$
    $\{15,10, 2 \}$
    $\{ 2, 7, 13 \}$ \\
    &
    $\{ 13, 14, 20 \}$ 
    $\{ 20,1, 2 \}$
    $\{ 2, 18,17 \}$
    $\{ 18,13, 14 \}$
    $\{ 14,1, 9 \}$
    $\{ 9, 13,10\}$
    $\{ 10,5, 16 \}$
    $\{ 16,9,7\}$ \\
&
    $\{ 7,19, 17 \}$
    $\{ 17,5, 15 \}$
    $\{ 15,7,1 \}$
    $\{ 1, 3, 10 \}$
    $\{ 10,6, 7\}$
    $\{ 7, 11,8 \}$
    $\{ 8,2,16 \}$
    $\{ 16,13,12\}$ \\
&
    $\{ 12,4, 17 \}$
    $\{ 17, 18, 20 \}$
    $\{20, 9, 10 \}$
    $\{ 10,4, 18 \}$
    $\{ 18,13,5 \}$
    $\{ 5, 6, 9 \}$
    $\{ 9,4, 15 \}$
    $\{ 15, 20,16 \}$ \\
&
    $\{ 16,3, 11 \}$
    $\{ 11,10, 14 \}$
    $\{ 14,16,4 \}$
    $\{ 4,7,3 \}$
    $\{ 3,19, 18 \}$
    $\{ 18,6, 16 \}$
    $\{ 16,1, 17 \}$
    $\{ 17,11,6 \}$ \\
&
    $\{ 6,13,4 \}$
    $\{ 4,3, 20 \}$
    $\{ 20,19,18 \}$
    $\{ 18,7, 12 \}$
    $\{ 12,2, 14 \}$
    $\{ 14,7,5\}$
    $\{ 5,2,1\}$
    $\{ 1, 6, 12 \}$ \\
&
    $\{ 12,8, 9 \}$
    $\{ 9, 18,11 \}$
    $\{ 11, 12, 20 \}$
    $\{ 20,5, 6 \}$
    $\{ 6, 14, 19 \}$
    $\{ 19,16,15\}$
    $\{ 15,3, 13 \}$ 
    $\{13,11, 1 \}$ \\
    &
    $\{ 1, 4, 19 \}$
    $\{ 19,10, 12\}$ \\
\hline
\end{longtable}
}

\subsection{\boldmath Some small c.a.h. $\{3\}$-GDDs} \hfill
\label{3GDDa}

{\footnotesize
\begin{longtable}{|c|l|}
\hline
$T$ & c.a.h. cycle in $\{3\}$-GDD of type $T$ \\
\hline
$2^4$ & $\{2,3,1\}$ 
$\{1,8,7\}$ 
$\{7,2,4\}$ 
$\{4,1,6\}$ 
$\{6,7,5\}$ 
$\{5,4,3\}$ 
$\{3,6,8\}$ 
$\{8,5,2\}$ \\
\hline
$1^6 3^1$ & $\{8,6,4\}$ 
$\{4,1,7\}$ 
$\{7,9,2\}$ 
$\{2,4,3\}$ 
$\{3,6,9\}$ 
$\{9,8,1\}$ 
$\{1,3,5\}$ 
$\{5,8,2\}$ 
$\{2,1,6\}$ 
$\{6,5,7\}$ \\
&
$\{7,3,8\}$ \\
\hline
$3^3$ & $\{5,6,7\}$ 
$\{7,3,8\}$ 
$\{8,1,9\}$ 
$\{9,7,2\}$ 
$\{2,4,3\}$ 
$\{3,5,1\}$ 
$\{1,2,6\}$ 
$\{6,8,4\}$ 
$\{4,9,5\}$ \\
\hline
$2^3 4^1$ & $\{4,2,1\}$ 
$\{1,8,3\}$ 
$\{3,2,9\}$ 
$\{9,6,1\}$ 
$\{1,7,10\}$ 
$\{10,5,2\}$ 
$\{2,7,8\}$ 
$\{8,5,6\}$ 
$\{6,10,3\}$ \\
&
$\{3,4,5\}$ 
$\{5,9,7\}$ 
$\{7,6,4\}$ \\
\hline
$2^6$ & $\{3,2,1\}$ 
$\{1,5,4\}$ 
$\{4,2,6\}$ 
$\{6,8,1\}$ 
$\{1,10,9\}$ 
$\{9,2,11\}$ 
$\{11,1,12\}$ 
$\{12,10,2\}$ 
$\{2,5,7\}$ \\
&
$\{7,4,3\}$ 
$\{3,5,12\}$ 
$\{12,9,4\}$ 
$\{4,8,11\}$ 
$\{11,3,6\}$ 
$\{6,7,9\}$ 
$\{9,8,5\}$ 
$\{5,6,10\}$ 
$\{10,11,7\}$ \\
&
$\{7,12,8\}$ 
$\{8,10,3\}$ \\
\hline
$1^{12}3^1$ & $\{11,15,5\}$ 
$\{5,13,1\}$ 
$\{1,9,12\}$ 
$\{12,3,5\}$ 
$\{5,14,9\}$ 
$\{9,15,2\}$ 
$\{2,10,13\}$ 
$\{13,12,7\}$ \\
&
$\{7,6,1\}$ 
$\{1,3,2\}$ 
$\{2,8,12\}$ 
$\{12,6,14\}$ 
$\{14,15,13\}$ 
$\{13,4,3\}$ 
$\{3,15,7\}$ 
$\{7,8,9\}$ \\
&
$\{9,4,10\}$ 
$\{10,8,5\}$ 
$\{5,7,2\}$ 
$\{2,14,4\}$ 
$\{4,12,15\}$ 
$\{15,1,10\}$ 
$\{10,12,11\}$ 
$\{11,3,9\}$ \\
&
$\{9,6,13\}$ 
$\{13,11,8\}$ 
$\{8,15,6\}$ 
$\{6,5,4\}$ 
$\{4,8,1\}$ 
$\{1,14,11\}$ 
$\{11,4,7\}$ 
$\{7,14,10\}$ \\
&
$\{10,3,6\}$ 
$\{6,2,11\}$ \\
\hline
$5^3$ & $\{2,3,1\}$ 
$\{1,11,6\}$ 
$\{6,4,5\}$ 
$\{5,9,1\}$ 
$\{1,8,15\}$ 
$\{15,2,4\}$ 
$\{4,9,14\}$ 
$\{14,1,12\}$ 
$\{12,8,4\}$ \\
&
$\{4,11,3\}$ 
$\{3,5,7\}$ 
$\{7,8,9\}$ 
$\{9,10,2\}$ 
$\{2,12,7\}$ 
$\{7,14,6\}$ 
$\{6,10,8\}$ 
$\{8,13,3\}$ \\
&
$\{3,10,14\}$ 
$\{14,15,13\}$ 
$\{13,5,12\}$ 
$\{12,11,10\}$ 
$\{10,5,15\}$ 
$\{15,7,11\}$ 
$\{11,9,13\}$ \\
&
$\{13,6,2\}$ \\
\hline
\end{longtable}
}

\section{\boldmath Some GDDs}

\subsection{\boldmath Some small $\{4\}$-GDDs of type $3^t$} \hfill
\label{4GDD3^t}

Each of the $\{4\}$-GDDs of type $3^t$ listed in this section is
on the set of points $[3t]$, with groups $\{i,i+t,i+2t\}$, $i\in[t]$.

{\footnotesize
\begin{longtable}{|c|l|}
\hline
$T$ & $\{4\}$-GDD of type $T$ \\
\hline
$3^4$ & {\boldmath $\{1,2,3,4\}$ $\{1,6,7,8\}$
$\{2,8,9,11\}$ $\{3,5,8,10\}$ }$\{4,5,6,11\}$ $\{4,7,9,10\}$
$\{2,5,7,12\}$ \\
& $\{1,10,11,12\}$ $\{3,6,9,12\}$ \\
\hline
$3^5$ & {\boldmath $\{1,2,4,8\}$ $\{2,3,5,9\}$
$\{3,4,6,10\}$ $\{4,5,7,11\}$ $\{5,6,8,12\}$ $\{6,7,9,13\}$} \\
& {\boldmath $\{7,8,10,14\}$} $\{2,10,11,13\}$ $\{1,9,10,12\}$
$\{3,11,12,14\}$ $\{4,12,13,15\}$ $\{1,5,13,14\}$ \\
& $\{2,6,14,15\}$ $\{1,3,7,15\}$ $\{8,9,11,15\}$ \\
\hline
$3^8$ & {\boldmath $\{9,14,15,19\}$
$\{1,2,5,19\}$ $\{1,3,8,12\}$ $\{2,3,6,20\}$ $\{2,4,13,16\}$
$\{3,4,7,21\}$} \\
& {\boldmath $\{4,5,9,22\}$ $\{5,6,10,23\}$
$\{1,6,7,11\}$ $\{11,17,18,21\}$} $\{2,7,9,12\}$
$\{3,9,10,13\}$ \\
& $\{4,10,11,14\}$ $\{5,11,12,15\}$ $\{6,12,13,17\}$
$\{7,13,14,18\}$ $\{10,15,17,20\}$ $\{12,18,19,22\}$ \\
&
$\{13,19,20,23\}$ $\{1,14,20,21\}$ $\{2,15,21,22\}$
$\{3,17,22,23\}$ $\{1,4,18,23\}$ $\{5,8,18,20\}$ \\
& $\{4,17,19,24\}$
$\{3,15,16,18\}$ $\{2,8,14,17\}$ $\{1,13,15,24\}$
$\{12,14,16,23\}$ $\{2,11,23,24\}$ \\
&  $\{1,10,16,22\}$
$\{8,9,21,23\}$ $\{7,20,22,24\}$ $\{6,16,19,21\}$ $\{8,11,13,22\}$
$\{10,12,21,24\}$ \\
&  $\{9,11,16,20\}$ $\{7,8,10,19\}$ $\{6,9,18,24\}$
$\{5,7,16,17\}$ $\{4,6,8,15\}$ $\{3,5,14,24\}$ \\
\hline
$3^9$ & {\boldmath $\{1,16,18,22\}$ $\{3,16,19,24\}$ $\{8,10,13,24\}$ $\{2,13,15,25\}$ $\{2,4,19,23\}$} \\
&
{\boldmath $\{1,4,15,26\}$ $\{3,15,20,23\}$ $\{3,4,5,7\}$ $\{2,5,6,21\}$ $\{1,2,7,17\}$ $\{3,10,14,17\}$} \\
& {\boldmath $\{4,9,10,25\}$
$\{7,8,9,12\}$$\{5,12,19,26\}$$\{11,12,14,24\}$}
 $\{12,13,16,23\}$ \\
 & $\{1,13,14,20\}$ $\{2,9,14,16\}$
$\{14,21,22,25\}$ $\{2,3,8,22\}$ $\{3,18,25,26\}$ $\{7,11,18,23\}$ \\
& $\{7,10,15,21\}$ $\{9,11,15,17\}$
$\{4,6,11,16\}$ $\{5,8,15,16\}$ $\{8,11,19,25\}$ $\{9,19,20,21\}$ \\
& $\{9,22,23,26\}$ $\{7,20,22,24\}$
$\{17,23,24,25\}$ $\{6,8,18,20\}$ $\{4,12,17,20\}$ $\{2,10,12,18\}$ \\
& $\{10,16,20,26\}$ $\{5,10,11,22\}$
$\{4,18,21,24\}$ $\{11,13,21,26\}$ $\{6,7,14,26\}$ $\{14,15,18,19\}$ \\
& $\{6,17,19,22\}$ $\{1,6,12,25\}$
$\{1,8,21,23\}$ $\{1,5,9,24\}$ $\{3,6,9,13\}$ $\{5,13,17,18\}$ \\
& $\{7,13,19,27\}$ $\{12,15,22,27\}$
$\{6,10,23,27\}$ $\{2,24,26,27\}$ $\{4,8,14,27\}$ $\{1,3,11,27\}$ \\
& $\{16,17,21,27\}$ $\{5,20,25,27\}$\\
\hline
$3^{12}$ & {\boldmath $\{1,8,12,19\}$
$\{1,2,5,10\}$ $\{2,3,6,11\}$ $\{3,4,7,13\}$ $\{4,5,8,14\}$
$\{5,6,9,15\}$} \\
& {\boldmath $\{6,7,10,16\}$ $\{7,8,11,17\}$
$\{8,9,13,18\}$ $\{9,10,14,19\}$ $\{2,9,20,24\}$} \\
&
{\boldmath $\{10,11,15,20\}$ $\{11,13,16,21\}$ $\{13,14,17,22\}$
$\{14,15,18,23\}$ $\{15,16,19,25\}$} \\
&
{\boldmath $\{16,17,20,26\}$
$\{17,18,21,27\}$ $\{18,19,22,28\}$ $\{19,20,23,29\}$
$\{20,21,25,30\}$} \\
& {\boldmath $\{21,22,26,31\}$ $\{22,23,27,32\}$
$\{23,25,28,33\}$ $\{25,26,29,34\}$ $\{26,27,30,35\}$} \\
&
{\boldmath $\{1,27,28,31\}$ $\{2,28,29,32\}$$\{3,29,30,33\}$ $\{4,30,31,34\}$
$\{5,31,32,35\}$} \\ 
& {\boldmath $\{1,6,32,33\}$ $\{2,7,33,34\}$
$\{3,8,34,35\}$ $\{1,4,9,35\}$} $\{5,7,20,27\}$ $\{5,13,23,24\}$ \\
&
$\{4,6,19,26\}$ $\{4,11,12,22\}$ $\{3,5,18,25\}$ $\{3,10,21,36\}$
$\{2,4,17,23\}$ $\{1,3,16,22\}$ \\
& $\{9,11,25,31\}$ $\{9,17,28,36\}$
$\{8,10,23,30\}$ $\{8,16,24,27\}$  $\{7,9,22,29\}$
$\{7,12,15,26\}$ \\
& $\{6,8,21,28\}$ $\{6,14,25,36\}$
$\{14,16,29,35\}$  $\{12,14,21,32\}$ $\{13,15,28,34\}$ \\
&
$\{13,20,31,36\}$ $\{11,14,27,33\}$ $\{11,19,24,30\}$
$\{10,13,26,32\}$  $\{10,12,18,29\}$ \\
& $\{3,17,19,32\}$
$\{12,17,25,35\}$ $\{2,16,18,31\}$ $\{16,23,34,36\}$
$\{1,15,17,30\}$ \\
& $\{15,22,24,33\}$ $\{6,20,22,35\}$
$\{3,12,20,28\}$ $\{5,19,21,34\}$ $\{2,19,27,36\}$
$\{4,18,20,33\}$ \\
&  $\{1,18,24,26\}$ $\{3,9,23,26\}$
$\{6,12,23,31\}$ $\{2,8,22,25\}$ $\{5,22,30,36\}$ $\{1,7,21,23\}$ \\
&
$\{4,21,24,29\}$  $\{6,13,27,29\}$ $\{9,12,27,34\}$
$\{5,11,26,28\}$ $\{8,26,33,36\}$ $\{4,10,25,27\}$ \\
&
$\{7,24,25,32\}$  $\{9,16,30,32\}$ $\{2,12,13,30\}$
$\{8,15,29,31\}$ $\{1,11,29,36\}$ $\{7,14,28,30\}$ \\
&
$\{10,24,28,35\}$  $\{11,18,32,34\}$ $\{4,15,32,36\}$
$\{10,17,31,33\}$ $\{3,14,24,31\}$ $\{2,15,21,35\}$ \\
&
$\{7,18,35,36\}$  $\{1,14,20,34\}$ $\{6,17,24,34\}$
$\{13,19,33,35\}$  $\{5,12,16,33\}$ \\
\hline
\end{longtable}
}

\subsection{\boldmath Some small $\{4,7\}$-GDDs of type $3^t u^1$} \hfill
\label{47GDD}

Each of the $\{4,7\}$-GDDs of type $3^t u^1$ listed in this section is on the set of
points $[3t+u]$, with groups $\{3i-2,3i-1,3i\}$, $i\in[t]$, and $\{3t+1,3t+2,\ldots,3t+u\}$.

{\footnotesize
\begin{longtable}{|c|l|}
\hline
$T$ & $\{4,7\}$-GDD of type $T$ \\
\hline
$3^5 6^1$ & {\boldmath $\{1,4,14,16\}$ $\{2,11,13,16\}$ $\{5,7,15,16\}$ $\{1,7,11,17\}$ $\{5,10,13,17\}$} \\
& {\boldmath $\{6,8,15,17\}$
$\{6,9,12,16\}$ $\{3,8,10,16\}$} {\em \{2,9,14,17\} \{1,9,10,18\}} $\{3,4,12,17\}$ \\
& $\{4,9,13,19\}$
$\{4,11,15,18\}$ $\{8,12,13,18\}$ $\{2,6,7,18\}$ $\{3,5,14,18\}$ $\{1,5,8,19\}$ \\
& $\{7,12,14,19\}$
$\{2,10,15,19\}$ $\{3,6,11,19\}$ $\{1,6,13,20\}$ $\{4,7,10,20\}$ $\{8,11,14,20\}$ \\
& $\{2,5,12,20\}$
$\{3,9,15,20\}$ $\{1,12,15,21\}$ $\{6,10,14,21\}$ $\{5,9,11,21\}$ $\{2,4,8,21\}$ \\
&  $\{3,7,13,21\}$
 \\
\hline
$3^6 6^1$ & {\boldmath $\{10,13,18,20\}$ $\{3,8,11,20\}$ $\{1,7,11,16\}$ 
$\{3,13,16,19\}$ $\{5,8,18,19\}$} \\
& {\boldmath $\{2,5,10,14\}$
$\{6,11,14,19\}$ $\{2,4,9,19\}$ $\{1,12,17,19\}$ $\{1,5,15,20\}$} \\
& {\em \{1,4,18,21\} \{7,14,18,22\}}
$\{2,6,16,20\}$ $\{4,8,13,17\}$ $\{9,14,17,20\}$ $\{4,7,12,20\}$ \\
& $\{7,10,15,19\}$ $\{2,11,18,23\}$
$\{3,5,7,21\}$ $\{3,10,17,22\}$ $\{3,4,14,23\}$ $\{6,8,10,21\}$ \\
& $\{1,6,13,22\}$ $\{6,7,17,23\}$
$\{9,11,13,21\}$ $\{5,9,16,22\}$ $\{1,9,10,23\}$ $\{12,14,16,21\}$ \\
& $\{2,8,12,22\}$ $\{5,12,13,23\}$
$\{2,15,17,21\}$ $\{4,11,15,22\}$ $\{8,15,16,23\}$ $\{1,8,14,24\}$ \\
& $\{4,10,16,24\}$ $\{5,11,17,24\}$
$\{2,7,13,24\}$ $\{3,6,9,12,15,18,24\}$
\\
\hline
$3^{10} 12^1$ & {\boldmath $\{1,4,8,31\}$
$\{4,7,11,32\}$ $\{7,16,29,31\}$ $\{3,19,25,31\}$ $\{2,10,19,32\}$} \\
& {\boldmath $\{2,11,18,31\}$ $\{12,15,24,31\}$ $\{13,24,25,32\}$
$\{6,22,28,32\}$} \\
& {\boldmath $\{14,17,28,31\}$ $\{10,21,22,31\}$
$\{5,14,21,32\}$ $\{1,17,20,32\}$ $\{9,20,26,31\}$} \\
&
{\boldmath $\{8,26,30,32\}$ $\{12,23,29,32\}$ $\{15,18,27,32\}$
$\{6,13,30,31\}$} {\em \{6,12,19,33\}} \\
& {\em \{2,6,14,34\}}
$\{5,23,27,31\}$ $\{3,9,16,32\}$ $\{7,10,14,33\}$ $\{1,9,25,33\}$
$\{5,13,22,33\}$ \\
& $\{16,27,28,33\}$ $\{4,20,23,33\}$
$\{3,11,29,33\}$  $\{2,15,26,33\}$  $\{8,17,24,33\}$ \\
&
$\{18,21,30,33\}$ $\{10,13,17,34\}$ $\{4,12,28,34\}$
$\{8,16,25,34\}$  $\{1,19,30,34\}$ \\
& $\{7,23,26,34\}$
$\{9,15,22,34\}$ $\{5,18,29,34\}$ $\{11,20,27,34\}$
$\{3,21,24,34\}$ \\
& $\{13,16,20,35\}$  $\{1,7,15,35\}$
$\{11,19,28,35\}$ $\{3,4,22,35\}$ $\{10,26,29,35\}$ \\
&
$\{12,18,25,35\}$  $\{5,9,17,35\}$  $\{2,8,21,35\}$
$\{14,23,30,35\}$ $\{6,24,27,35\}$ $\{16,19,23,36\}$ \\
&
$\{4,10,18,36\}$  $\{1,14,22,36\}$  $\{6,7,25,36\}$
$\{2,13,29,36\}$ $\{15,21,28,36\}$ $\{8,12,20,36\}$ \\
&
$\{5,11,24,36\}$  $\{3,17,26,36\}$  $\{9,27,30,36\}$
$\{19,22,26,37\}$ $\{7,13,21,37\}$ $\{4,17,25,37\}$ \\
&
$\{9,10,28,37\}$  $\{2,5,16,37\}$  $\{1,18,24,37\}$
$\{11,15,23,37\}$ $\{8,14,27,37\}$ $\{6,20,29,37\}$ \\
&
$\{3,12,30,37\}$  $\{22,25,29,38\}$  $\{10,16,24,38\}$
$\{7,20,28,38\}$ $\{1,12,13,38\}$ $\{5,8,19,38\}$ \\
& $\{4,21,27,38\}$
$\{14,18,26,38\}$  $\{11,17,30,38\}$ $\{2,9,23,38\}$
$\{3,6,15,38\}$ $\{2,25,28,39\}$ \\
& $\{13,19,27,39\}$
$\{1,10,23,39\}$  $\{4,15,16,39\}$ $\{8,11,22,39\}$
$\{7,24,30,39\}$ \\
& $\{17,21,29,39\}$  $\{3,14,20,39\}$
$\{5,12,26,39\}$  $\{6,9,18,39\}$ $\{1,5,28,40\}$
$\{16,22,30,40\}$ \\
& $\{4,13,26,40\}$ $\{7,18,19,40\}$
$\{11,14,25,40\}$  $\{3,10,27,40\}$ $\{2,20,24,40\}$ \\
&
$\{6,17,23,40\}$ $\{8,15,29,40\}$ $\{9,12,21,40\}$
$\{1,27,29,41\}$  $\{2,4,30,41\}$ $\{3,5,7,41\}$ \\
&  $\{6,8,10,41\}$
$\{9,11,13,41\}$ $\{12,14,16,41\}$ $\{15,17,19,41\}$
$\{18,20,22,41\}$ \\
& $\{21,23,25,41\}$  $\{24,26,28,41\}$
$\{1,6,11,16,21,26,42\}$ $\{2,7,12,17,22,27,42\}$ \\
&
$\{3,8,13,18,23,28,42\}$ $\{4,9,14,19,24,29,42\}$
$\{5,10,15,20,25,30,42\}$ 
 \\
\hline
\end{longtable}
}

\section*{Acknowledgment}

The authors are grateful 
to Megan Dewar for making available her PhD thesis \cite{Dewar:2007}
and to Jerzy Jaromczyk for making
available the paper \cite{Gilkersonetal:unpublished}.

\providecommand{\bysame}{\leavevmode\hbox to3em{\hrulefill}\thinspace}
\providecommand{\MR}{\relax\ifhmode\unskip\space\fi MR }
\providecommand{\MRhref}[2]{%
  \href{http://www.ams.org/mathscinet-getitem?mr=#1}{#2}
}
\providecommand{\href}[2]{#2}


\begin{thebibliography}{10}

\bibitem{AggarwalVitter:1988}
A.~Aggarwal and J.~S. Vitter, \emph{The input/output complexity of sorting and
  related problems}, Commun. ACM \textbf{31} (1988), no.~9, 1116--1127.

\bibitem{Alspachetal:1990}
B.~Alspach, K.~Heinrich, and B.~Mohar, \emph{A note on {H}amilton cycles in
  block-intersection graphs}, Finite geometries and combinatorial designs
  ({L}incoln, {NE}, 1987), Contemp. Math., vol. 111, Amer. Math. Soc.,
  Providence, RI, 1990, pp.~1--4.

\bibitem{Bitneretal:1976}
J.~R. Bitner, G.~Ehrlich, and E.~M. Reingold, \emph{Efficient generation of the
  binary reflected {G}ray code and its applications}, Comm. ACM \textbf{19}
  (1976), no.~9, 517--521.

\bibitem{BlackburnMckee:2010}
S.~R. Blackburn and J.~F. Mckee, \emph{Constructing $k$-radius sequences},
  arXiv:1006.5812v1 (2010).

\bibitem{Brouweretal:1977}
A.~E. Brouwer, A.~Schrijver, and H.~Hanani, \emph{Group divisible designs with
  block-size four}, Discrete Math. \textbf{20} (1977), no.~1, 1--10.

\bibitem{Chungetal:1992}
F.~Chung, P.~Diaconis, and R.~Graham, \emph{Universal cycles for combinatorial
  structures}, Discrete Math. \textbf{110} (1992), no.~1--3, 43--59.

\bibitem{Colbournetal:1992b}
C.~J. Colbourn, D.~G. Hoffman, and R.~Rees, \emph{A new class of group
  divisible designs with block size three}, J. Combin. Theory Ser. A
  \textbf{59} (1992), no.~1, 73--89.

\bibitem{ColbournRosa:1999}
C.~J. Colbourn and A.~Rosa, \emph{Triple {S}ystems}, Oxford Mathematical
  Monographs, Oxford University Press, New York, 1999.

\bibitem{deBruijn:1946}
N.~G. de~Bruijn, \emph{A combinatorial problem}, Nederl. Akad. Wetensch., Proc.
  \textbf{49} (1946), 758--764 = Indagationes Math. 8, 461--467 (1946).

\bibitem{Dewar:2007}
M.~Dewar, \emph{Gray codes, universal cycles and configuration orderings for
  block designs}, Ph.D. thesis, Carleton University, Ottawa, Ontario, Canada,
  2007.

\bibitem{EadesMcKay:1984}
P.~Eades and B.~McKay, \emph{An algorithm for generating subsets of fixed size
  with a strong minimal change property}, Inform. Process. Lett. \textbf{19}
  (1984), no.~3, 131--133.

\bibitem{FortHedlund:1958}
M.~K. Fort, Jr. and G.~A. Hedlund, \emph{Minimal coverings of pairs by
  triples}, Pacific J. Math. \textbf{8} (1958), 709--719.

\bibitem{Ghosh:1975}
S.~P. Ghosh, \emph{Consecutive storage of relevant records with redundancy},
  Commun. ACM \textbf{18} (1975), no.~8, 464--471.

\bibitem{GilkersonJaromczyk:2002}
J.~W. Gilkerson and J.~W. Jaromczyk, \emph{Restoring the order of images in a
  sequence of {MRI} slices}, unpublished manuscript (2002).

\bibitem{Gilkersonetal:unpublished}
J.~W. Gilkerson, J.~W. Jaromczyk, and Z.~Lonc, \emph{On constructing sequences
  of radius $k$ using finite geometries}, unpublished manuscript.

\bibitem{Gray:1953}
F.~Gray, \emph{Pulse code communication}, 1953 (file 1947), {US} Patent
  2,632,058.

\bibitem{Gronauetal:1995}
H.-D. O.~F. Gronau, R.~C. Mullin, and Ch. Pietsch, \emph{The closure of all
  subsets of {$\{3,4,\cdots,10\}$} which include {$3$}}, Ars Combin.
  \textbf{41} (1995), 129--161.

\bibitem{Hare:1995}
D.~R. Hare, \emph{Cycles in the block-intersection graph of pairwise balanced
  designs}, Discrete Math. \textbf{137} (1995), no.~1-3, 211--221.

\bibitem{Horaketal:1999}
P.~Hor{\'a}k, D.~A. Pike, and M.~E. Raines, \emph{Hamilton cycles in
  block-intersection graphs of triple systems}, J. Combin. Des. \textbf{7}
  (1999), no.~4, 243--246.

\bibitem{HorakRosa:1988}
P.~Hor{\'a}k and A.~Rosa, \emph{Decomposing {S}teiner triple systems into small
  configurations}, Ars Combin. \textbf{26} (1988), 91--105.

\bibitem{Hurlbert:1994}
G.~Hurlbert, \emph{On universal cycles for {$k$}-subsets of an {$n$}-set}, SIAM
  J. Discrete Math. \textbf{7} (1994), no.~4, 598--604.

\bibitem{Jackson:unpublished}
B.~Jackson, \emph{Universal cycles of $4$-subsets and $5$-subsets}, unpublished
  manuscript.

\bibitem{Jackson:1993}
B.~W. Jackson, \emph{Universal cycles of {$k$}-subsets and {$k$}-permutations},
  Discrete Math. \textbf{117} (1993), no.~1--3, 141--150.

\bibitem{JaromczykLonc:2004}
J.~W. Jaromczyk and Z.~Lonc, \emph{Sequences of radius {$k$}: how to fetch many
  huge objects into small memory for pairwise computations}, Algorithms and
  Computation, Lecture Notes in Comput. Sci., vol. 3341, Springer, Berlin,
  2004, pp.~594--605.

\bibitem{Lenz:1984}
H.~Lenz, \emph{Some remarks on pairwise balanced designs}, Mitt. Math. Sem.
  Giessen (1984), no.~165, 49--62.

\bibitem{Ruskey:1988}
F.~Ruskey, \emph{Adjacent interchange generation of combinations}, J.
  Algorithms \textbf{9} (1988), no.~2, 162--180.

\bibitem{Savage:1997}
C.~Savage, \emph{A survey of combinatorial {G}ray codes}, SIAM Rev. \textbf{39}
  (1997), no.~4, 605--629.

\bibitem{Senetal:2002}
S.~Sen, S.~Chatterjee, and N.~Dumir, \emph{Towards a theory of cache-efficient
  algorithms}, J. Assoc. Comput. Mach. \textbf{49} (2002), no.~6, 828--858.

\bibitem{Storme:2007}
L.~Storme, \emph{Finite geometry}, The {CRC} Handbook of Combinatorial Designs
  (C.~J. Colbourn and J.~H. Dinitz, eds.), Chapman \& Hall, Boca Raton, FL,
  2007, pp.~702--729.

\bibitem{Wilson:1972a}
R.~M. Wilson, \emph{An existence theory for pairwise balanced designs. {I}.
  {C}omposition theorems and morphisms}, J. Combin. Theory Ser. A \textbf{13}
  (1972), 220--245.

\end{thebibliography}
\end{document}